\documentclass[review,onefignum,onetabnum]{siamonline220329}

\usepackage{lipsum}
\usepackage{amsfonts}
\usepackage{graphicx}
\usepackage{epstopdf}
\usepackage{algorithmic}
\usepackage{xcolor}

\def\sc{\scriptsize}
\def\sl{\small}
\def\ds{\displaystyle}

\ifpdf
  \DeclareGraphicsExtensions{.eps,.pdf,.png,.jpg}
\else
  \DeclareGraphicsExtensions{.eps}
\fi

\graphicspath{{Images/}}

\newcommand{{\sLP}}{\ensuremath{ {\mbox{\tiny{LP}}} }}
\usepackage{enumitem}
\setlist[enumerate]{leftmargin=.5in}
\setlist[itemize]{leftmargin=.5in}

\newsiamremark{remark}{Remark}
\newsiamremark{hypothesis}{Hypothesis}
\crefname{hypothesis}{Hypothesis}{Hypotheses}
\newsiamthm{claim}{Claim}

\headers{Obligate mutualism in a chemostat}{Mtar and Fekih-Salem}

\title{Global dynamics and bifurcation analysis of a chemostat model with obligate mutualism and mortality
\thanks{Submitted to the editor 2026-03-11.
}}

\author{Tahani Mtar\thanks{University of Tunis El Manar, National Engineering School of Tunis, LAMSIN, Tunisia (\email{tahani.mtar@enit.utm.tn}).}
\and Radhouane Fekih-Salem\footnotemark[1]~\thanks{University of Tunis El Manar, National Engineering School of Tunis, LAMSIN, Tunisia    (\email{radhouane.fekih-salem@enit.utm.tn}).}}

\usepackage{amsopn}

\ifpdf
\hypersetup{
  pdftitle={Obligate mutualism},
  pdfauthor={MTAR, FEKIH}
}
\fi

\begin{document}

\maketitle
\begin{abstract}
We propose a system of differential equations modeling the competition between two obligate mutualistic species for a single nutrient in a chemostat. Each species promotes the growth of the other, and growth occurs only in the presence of its partner. The three-dimensional model incorporates interspecific density-dependent growth functions and distinct removal rates. We perform a mathematical analysis by characterizing the multiplicity of equilibria and deriving conditions for their existence and stability. Using \textsc{MatCont}, we construct numerical operating diagrams in the parameter space of dilution rate and input substrate concentration, providing a global view of the qualitative dynamics of the system. One-parameter bifurcation diagrams with respect to the input substrate then reveal a variety of dynamical transitions, including saddle–node, Hopf, limit point of cycles (LPC), period-doubling (PD), and homoclinic bifurcations. When mortality is included, the system exhibits a richer dynamical repertoire than in the mortality-free case, with stable and unstable periodic orbits, tri-stability between equilibria and limit cycles, and several codimension-two bifurcations, including Bogdanov–Takens (BT), cusp of cycles (CPC), resonance points (R1 and R2), and generalized Hopf (GH) points. These features allow coexistence not only around positive equilibria but also along stable limit cycles, reflecting more realistic ecological dynamics. In contrast, neglecting mortality restricts coexistence to equilibria only. Overall, this study highlights the critical role of mortality in shaping complex dynamics in obligate mutualism, producing multistability and oscillatory coexistence patterns that may better represent natural microbial or ecological systems.
\end{abstract}
\begin{keywords}
Codimension-two bifurcations, Generalized Hopf, Homoclinic bifurcation, Limit Point of Cycles, \textsc{MatCont}.
\end{keywords}
\begin{MSCcodes}
  34A34, 34D20, 37N25, 92B05
\end{MSCcodes}
\section{Introduction}	
The chemostat has long served as a fundamental paradigm in mathematical biology and ecology to study microbial growth and species interactions under controlled conditions. 
Building on the pioneering works of Monod \cite{Monod1949} and Novick–Szilard \cite{Novick1950}, it provides a framework to understand how competition, cooperation, and coexistence emerge in nutrient-limited environments.  

Classical chemostat models have revealed key insights into competitive exclusion, persistence, and oscillatory dynamics \cite{HarmandBook2017,SmithBook1995}.  
According to the {\em Competitive Exclusion Principle} (CEP), when multiple species compete for a single limiting resource, only the species that survives at the lowest substrate concentration persists, while all others are excluded.  
However, this prediction contrasts with the biodiversity observed in microbial ecosystems, motivating the study of mechanisms that can promote coexistence.  

Several mechanisms have been proposed in chemostat models to reconcile the CEP with microbial diversity, including flocculation and wall adhesion \cite{FekihJMAA2013,FekihAMM2016,FekihSIADS2019,FekihJBS2025,PilyuginSIAP1999,SARISIADS2026}, 
intra- and interspecific interference \cite{AbdellatifMBE2016,FekihJMAA2025,FekihMB2017}, 
predator-prey or multi-trophic interactions \cite{MtarIJB2021,MtarDcdsB2022,WadeJTB2016}.  

Among these, a natural question arises regarding the role of mutualistic interactions, in which each species enhances the growth of its partner. 
Specifically, it remains to be clarified how mutualism affects coexistence, stability, and the possible emergence of oscillatory or complex dynamics in nutrient-limited chemostat systems. 

In this work, we study two microbial populations interacting through both competition for a single nutrient and mutualistic cooperation. 
Persistence of each species requires the presence of its partner, and we account for distinct removal rates and mortality. 
The dynamics are described by the following system of ordinary differential equations:
\begin{equation}                                \label{ModelYi}
\left\{\begin{array}{lll}
\dot{S}   &=& D(S_{in}-S) - \mu_1(S,X_2)X_1 - \mu_2(S,X_1)X_2,\\[0.3em]
\dot{X}_1 &=& (Y_1\mu_1(S,X_2)-D_1)X_1, \\[0.3em]
\dot{X}_2 &=& (Y_2\mu_2(S,X_1)-D_2)X_2,
\end{array}\right.
\end{equation}
where $S(t)$ denotes the nutrient concentration, $X_i(t)$ the population concentrations, $S_{in}$ the input concentration, $D$ the dilution rate, and $Y_i$ the yield coefficients.
For $i\neq j$, $\mu_i(S,X_j)$ denotes the growth rate of species $X_i$, assumed to be increasing in $S$ and decreasing in the competitor density $X_j$. 
The removal rates $D_i$ are modeled as in \cite{FekihSIADS2021,FekihSIADS2019,Sari2024AIMS} by
$$
D_i = \alpha_i D + a_i, \qquad i=1,2,
$$
where $a_i \ge 0$ represents the specific mortality of species $X_i$, and $\alpha_i\in[0,1]$ decouples the hydraulic retention time $\mathrm{HRT}=1/D$ from the solid retention time $\mathrm{SRT}=1/(\alpha_i D)$ \cite{BenyahiaJPC2012,SariProcesses2022}. Thus, both $D_i>D$ and $D_i<D$ are biologically meaningful.
After rescaling $x_i = X_i/Y_i$, system \cref{ModelYi} becomes
\begin{equation}                             \label{ModelMutualism}
\left\{\begin{array}{lll}
\dot{S}   &=& D(S_{in}-S)- f_1(S,x_2)x_1-f_2(S,x_1)x_2,\\ 
\dot{x}_1 &=& (f_1(S,x_2)-D_1)x_1, \\ 
\dot{x}_2 &=& (f_2(S,x_1)-D_2)x_2,
\end{array}\right.
\end{equation}
where $f_1(S,x_2)=Y_1\,\mu_1(S,Y_2x_2)$ and $f_2(S,x_1)=Y_2\,\mu_2(S,Y_1x_1)$. 
The particular case $D_1 = D_2 = D$ was analyzed by El Hajji \cite{ElhajjiIJB2018,ElhajjiJBD2009}, who showed that coexistence occurs only at equilibria. 
By allowing distinct removal rates, our model significantly extends this framework and leads to a much richer dynamical behavior. In particular, we show that the system may undergo several local and global bifurcations, including saddle–node, Hopf, limit point of cycles (LPC), period-doubling (PD), and generalized Hopf (GH) bifurcations, giving rise to coexistence around stable periodic oscillations.

The main objective of this work is to provide a detailed mathematical and numerical investigation of model \cref{ModelMutualism} in order to understand how density-dependent competition and mutualistic interactions shape the dynamics of microbial populations in a chemostat. 
We derive necessary and sufficient conditions for the existence and local stability of equilibria with respect to the operating parameters $S_{in}$ and $D$. 
The global organization of the dynamics is then explored through operating diagrams and one-parameter bifurcation diagrams with respect to $S_{in}$, computed using the continuation software \textsc{MatCont} \cite{MATCONT2023}. 
These diagrams reveal a wide spectrum of dynamical regimes and bifurcation structures that illustrate the complexity induced by the interplay between competition, mutualism, and mortality.

This paper is organized as follows. 
In \Cref{Sec-AnalMod}, we introduce the assumptions on the growth functions and analyze the existence and local stability of the equilibria of model \cref{ModelMutualism}. 
\Cref{SectionHopfBif} presents numerical evidence of the Hopf bifurcation arising as the control parameter $S_{in}$ varies. 
The particular case without mortality ($D_1=D_2=D$) is then examined in \Cref{SubSec-EquilibSanMort}. A geometric analysis based on the intersection of isoclines and the monotonicity of the vector field in the regions they delimit yields simple criteria for the global stability of the washout equilibrium.
\Cref{Sec-OD} presents the operating diagram obtained numerically with \textsc{MatCont}, while \Cref{Sec-BD} investigates the one-parameter bifurcation diagram with respect to the input concentration $S_{in}$.
Finally, conclusions are drawn in \Cref{Sec-Conc}. Technical results, including the parameter values employed in the simulations, are reported in the Appendices.
\section{Assumptions on the model and mathematical analysis}           \label{Sec-AnalMod}
Consider system \cref{ModelMutualism}. 
For $i,j \in \{1,2\}$ with $i\neq j$, we assume that the growth function $f_i$ is continuously differentiable ($\mathcal{C}^1$) and satisfies the following hypotheses:
\begin{hypothesis}                                        \label{hyp1} 
$f_i(0,x_j)=0$ and $\ds{\frac{\partial f_i}{\partial S}(S,x_j)>0}$, for all $S \geq 0$ and $x_j> 0$.
\end{hypothesis}
\begin{hypothesis}                                        \label{hyp2}
$\ds{ \frac{\partial f_i}{\partial x_j}(S,x_j)>0}$, for all $S> 0$ and $x_j\geq 0$.
\end{hypothesis}
\begin{hypothesis}                                        \label{hyp3}
$f_i(S,0)=0$, for all $S\geq 0$.
\end{hypothesis}
\cref{hyp1} expresses that growth occurs only if substrate is present, and that growth rates increase with the substrate concentration. 
\cref{hyp2} reflects the mutualistic interaction, namely that each species enhances the growth of the other. 
Finally, \cref{hyp3} characterizes the obligate nature of the mutualism: nutrient uptake by one species requires the presence of the other.

We first establish two fundamental properties of chemostat models, which are crucial for the biological consistency of the analysis. More precisely, we show that solutions of system \cref{ModelMutualism} are positive and bounded for all $t \ge 0$. Using classical arguments similar to those in \cite{MtarDcdsB2022}, we obtain the following preliminary result.
\begin{proposition}
Assume that \cref{hyp1,hyp2,hyp3} hold.
For any nonnegative initial condition, the solution of system \cref{ModelMutualism} exists for all $t \geq 0$, remains nonnegative, and is bounded.
In addition, the set 
$$
\Omega = \left\{ (S,x_1,x_2) \in \mathbb{R}_+^3 : \; S+x_1+x_2 \leq \tfrac{D S_{in}}{D_{\min}} \right\}, 
\quad D_{\min} = \min(D,D_1,D_2),
$$
is positively invariant and is a global attractor for the dynamics of system \cref{ModelMutualism}.
\end{proposition}
In this section, we summarize the main results concerning the existence and stability of the equilibria of system \cref{ModelMutualism}.
We first determine the equilibria by setting the right-hand sides of system \cref{ModelMutualism} equal to zero.
They correspond to the nonnegative solutions of the following system:
\begin{equation}                                    \label{ModelSS}
\left\{\begin{array}{lll}
0   &=& D(S_{in}-S)- f_1(S,x_2)x_1-f_2(S,x_1)x_2,\\
0   &=& (f_1(S,x_2)-D_1)x_1, \\ 
0   &=& (f_2(S,x_1)-D_2)x_2.
\end{array}\right.
\end{equation}
From \cref{hyp3}, it follows that system \cref{ModelMutualism} cannot have an equilibrium at which one species is extinct while the other persists.
Indeed, if $x_1=0$, then the third equation of \cref{ModelSS} together with \cref{hyp3} implies that $x_2=0$.
Similarly, if $x_2=0$, then necessarily $x_1=0$.
Consequently, no single-species equilibria exist.

This property reflects the obligate nature of the mutualistic interaction: each species depends on the presence of the other for persistence, so extinction of one inevitably leads to extinction of its partner.
Therefore, system \cref{ModelMutualism} has only two types of equilibria:
\begin{itemize}
\item $\mathcal{E}_0=(S_{in},0,0)$: the washout equilibrium, at which both species are extinct;
\item $\mathcal{E}^*=(S^*,x_1^*,x_2^*)$: a coexistence equilibrium with $x_1^*>0$ and $x_2^*>0$.
\end{itemize}
Setting $x_1=x_2=0$ in system \cref{ModelSS}, the first equation yields $S=S_{in}$.
Hence, the washout equilibrium $\mathcal{E}_0=(S_{in},0,0)$ always exists.
Next, for the positive equilibrium $\mathcal{E}^*$, the components $S^*$, $x_1^*$, and $x_2^*$ must satisfy system \cref{ModelSS} with $x_1^*>0$ and $x_2^*>0$.
Equivalently, $(S^*,x_1^*,x_2^*)$ must satisfy
\begin{align}
D(S_{in}-S^*)  &= D_1x_1^*+D_2x_2^*,  \label{EquSolSet} \\
f_1(S^*,x_2^*) &= D_1,               \label{Equf1=D1} \\
f_2(S^*,x_1^*) &= D_2.               \label{Equf2=D2}
\end{align}
From \cref{EquSolSet}, we obtain the explicit expression for $S^*$:
\begin{equation}                        \label{ExpSetoil}
S^*= S_{in}-\frac{D_1}{D}x_1^*-\frac{D_2}{D}x_2^*.
\end{equation}
Substituting \cref{ExpSetoil} into \cref{Equf1=D1} and \cref{Equf2=D2} yields the following reduced system for $(x_1^*,x_2^*)$:
\begin{equation}                        \label{EqsExiEet}
\left\{
\begin{array}{ll}
f_1\!\left(S_{in}-\tfrac{D_1}{D}x_1-\tfrac{D_2}{D}x_2,\,x_2\right) = D_1,\\[1ex]
f_2\!\left(S_{in}-\tfrac{D_1}{D}x_1-\tfrac{D_2}{D}x_2,\,x_1\right) = D_2.
\end{array}
\right.
\end{equation}

A coexistence equilibrium $\mathcal{E}^*$ exists if and only if system \cref{EqsExiEet} has a solution in the interior $\mathring{M}$ of the set
\begin{equation}                              \label{SetM}
M := \left\{ (x_1,x_2) \in \mathbb{R}_+^2 : \frac{D_1}{D}x_1+\frac{D_2}{D}x_2 \leq S_{in} \right\}.
\end{equation}
In what follows, we denote by $\delta$ the boundary of $M$, defined by
$$
\delta:\quad \frac{D_1}{D}x_1+\frac{D_2}{D}x_2 = S_{in}.
$$
We also introduce the partial derivatives
\begin{equation}                           \label{ExprEFGH}
E= \frac{\partial f_1}{\partial S},\quad 
F= \frac{\partial f_2}{\partial S},\quad 
G= \frac{\partial f_1}{\partial x_2},\quad 
H= \frac{\partial f_2}{\partial x_1},
\end{equation}
all evaluated at the equilibrium $(S^*,x_1^*,x_2^*)$.
Under \cref{hyp1,hyp2}, these quantities are positive.

To solve system \cref{EqsExiEet} in the open set $\mathring{M}$, we first establish the following lemma.
\begin{lemma}                    \label{LemCondExis_xji}
Assume that \cref{hyp1,hyp2,hyp3} hold.
For $i,j\in\{1,2\}$ with $i\neq j$, the equation
\begin{equation}                     \label{EquExis_xji}
f_i\!\left(S_{in}-\tfrac{D_j}{D}x_j,\,x_j\right)=D_i
\end{equation}
has a solution in the interval $[0,\,D S_{in}/D_j]$ if and only if
\begin{equation}                    \label{CondExis_xji}
\max_{x_j \in [0,\,D S_{in}/D_j]}
f_i\!\left(S_{in}-\tfrac{D_j}{D}x_j,\,x_j\right)\geq D_i.
\end{equation}
\end{lemma}
\begin{proof}
For $i,j\in\{1,2\}$ with $i\neq j$, let $\phi_i$ be the function defined by
\begin{equation}                          \label{FunPhi_i}
\phi_i(x_j):= f_i\!\left(S_{in}-\tfrac{D_j}{D}x_j,\,x_j\right).
\end{equation}
Since $f_i$ is continuous by assumption, the function $\phi_i$ is continuous on the compact interval $[0,\,D S_{in}/D_j]$.
Moreover, from \cref{hyp3} and \cref{hyp1}, we have
\[
\phi_i(0)=f_i(S_{in},0)=0
\quad\text{and}\quad
\phi_i(D S_{in}/D_j)=f_i(0,D S_{in}/D_j)=0.
\]
Therefore, $\phi_i$ attains a maximum on $[0,\,D S_{in}/D_j]$.
It follows from the intermediate value theorem that equation \cref{EquExis_xji} has a solution in $[0,\,D S_{in}/D_j]$ if and only if
\[
\max_{x_j \in [0,\,D S_{in}/D_j]} \phi_i(x_j)\ge D_i,
\]
which is precisely condition \cref{CondExis_xji} (see \cref{Fig-NmbExSol_xji}).
\end{proof}
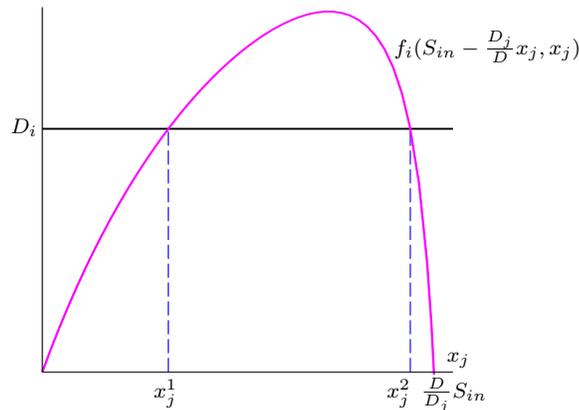
\begin{figure}[!ht]
\setlength{\unitlength}{1.0cm}
\begin{center}
\begin{picture}(6.3,5.5)(0,0)
\put(0,0){\rotatebox{0}{\includegraphics[width=6cm,height=5cm]{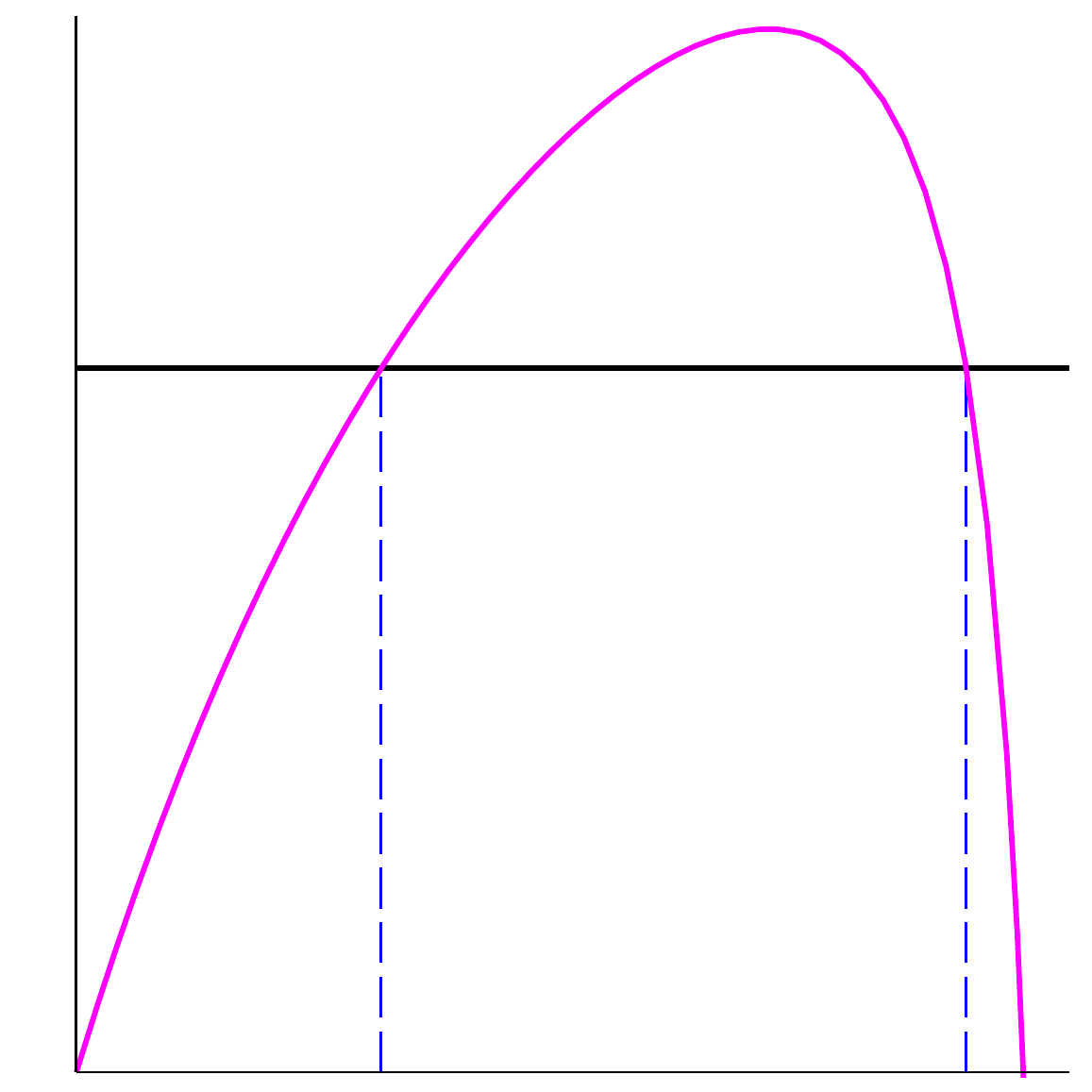}}}
\put(5.1,4.3){{\sc $f_i(S_{in}-\tfrac{D_j}{D}x_j,x_j)$}}
\put(5.8,0.23){{\sc  $x_j$}}
\put(0,3.25){{\sc  $D_i$}}
\put(5.,-0.25){{\sc $x_j^2$  }}
\put(5.45,-0.25){{\sc $\frac{D}{D_j}S_{in}$  }}
\put(1.9,-0.25){{\sc $x_j^1$  }}
\end{picture}
\end{center}
\caption{Graphical illustration of the existence of solutions of equation \cref{EquExis_xji}.
Depending on the parameter values, the horizontal line $y=D_i$ may intersect the curve
$x_j\mapsto f_i(S_{in}-\tfrac{D_j}{D}x_j,x_j)$ at zero, one, or two points.}
\label{Fig-NmbExSol_xji}
\end{figure}

To simplify the analysis, we introduce the following additional hypothesis, which is satisfied by the specific growth functions defined in \cref{SpeciFunc}.
\begin{hypothesis}                         \label{hyp4}
For $i,j\in\{1,2\}$ with $i\neq j$, equation \cref{EquExis_xji} has at most two solutions $x_j^1$ and $x_j^2$ in the interval $(0,\,D S_{in}/D_j)$.
\end{hypothesis}
\Cref{hyp4} characterizes the maximal number of solutions of equation \cref{EquExis_xji} for the specific growth functions introduced in \cref{SpeciFunc}.
The proof is provided in \cref{AppendixA}.
 
In the general case, when the mapping
\[
x_j \longmapsto f_i\!\left(S_{in}-\tfrac{D_j}{D}x_j,\,x_j\right)
\]
is multimodal, the analysis can be conducted using the same geometric arguments as in the unimodal case.
In particular, the possible multiplicity of solutions does not introduce any additional conceptual difficulty.
To establish the existence of a positive equilibrium $\mathcal{E}^*=(S^*,x_1^*,x_2^*)$, we define the curves associated with equations \cref{Equf1=D1} and \cref{Equf2=D2}, as detailed in the following result.
The proof relies on arguments similar to those used in \cite[Lemmas 3.2 and 3.4]{MtarIJB2021}
\begin{lemma}                                   \label{LemFi}
Assume that \cref{hyp1,hyp2,hyp3,hyp4} hold and that condition \cref{CondExis_xji} is satisfied for $i,j\in\{1,2\}$, $i\neq j$. Let $x_j^1$ and $x_j^2$ denote the solutions of the equation $f_i(S_{in}-\tfrac{D_j}{D}x_j,x_j)=D_i$.
\begin{enumerate}[leftmargin=*] 
\item The equation
$$
f_1\!\left(S_{in}-\tfrac{D_1}{D}x_1-\tfrac{D_2}{D}x_2,\,x_2\right)=D_1
$$
defines a smooth function
$$
\begin{array}{ccccl}
F_1  & :  &  [x_2^1,x_2^2]  & \longrightarrow & [0,\, D S_{in}/D_1), \\
     &    &   x_2           & \longmapsto     &  F_1(x_2)=x_1,
\end{array}
$$
such that $F_1(x_2^1)=F_1(x_2^2)=0$. Moreover, the graph $\gamma_1$ of $F_1$ lies in
$\mathring{M}$, the interior of $M$, that is,
$(F_1(x_2), x_2) \in \mathring{M}$ for all $x_2 \in [x_2^1,x_2^2]$ (see \cref{Chap5-FigNull_MulEet}).
The derivative of $F_1$ is given by
\begin{equation}                            \label{ExpF1prim}
F'_1(x_2)=\frac{-D_2E+ DG}{D_1 E}, \quad \text{for all} \quad x_2 \in [x_2^1,x_2^2],
\end{equation}
where $E=\partial f_1/\partial S$ and $G=\partial f_1/\partial x_2$ are evaluated at
$(S,x_2)=(S_{in}-\tfrac{D_1}{D}F_1(x_2)-\tfrac{D_2}{D}x_2,x_2)$.
\item The equation 
\[
f_2\!\left(S_{in}-\tfrac{D_1}{D}x_1-\tfrac{D_2}{D}x_2,\,x_1\right)=D_2
\]
defines a smooth function
$$
\begin{array}{ccccl}
F_2  & :  &  [x_1^1,x_1^2]  & \longrightarrow & [0,\,D S_{in}/D_2), \\
     &    &   x_1             & \longmapsto     &  F_2(x_1)=x_2,
\end{array}
$$
such that $F_2(x_1^1)=F_2(x_1^2)=0$ and the graph $\gamma_2$ of $F_2$ lies in $\mathring{M}$, the interior of $M$, that is, $(x_1, F_2(x_1)) \in \mathring{M}$ for all $x_1 \in [x_1^1,x_1^2]$ (see \cref{Chap5-FigNull_MulEet}). The derivative of $F_2$ is given by
\begin{equation}                            \label{ExpF2prim}
F'_2(x_1)=\frac{-D_1 F+ DH}{D_2 F}, \quad \text{for all} \quad x_1 \in [x_1^1,x_1^2],
\end{equation}
where $F=\partial f_2/\partial S$ and $H=\partial f_2/\partial x_1$ are evaluated at $(S,x_1)=(S_{in}-\tfrac{D_1}{D}x_1-\tfrac{D_2}{D}F_2(x_1), x_1)$.
\end{enumerate}
\end{lemma}
\begin{figure}[!ht]
\setlength{\unitlength}{1.0cm}
\begin{center}
\begin{picture}(6.5,6.2)(0,0)
\put(-3.6,0){\rotatebox{0}{\includegraphics[width=6.5cm,height=6cm]{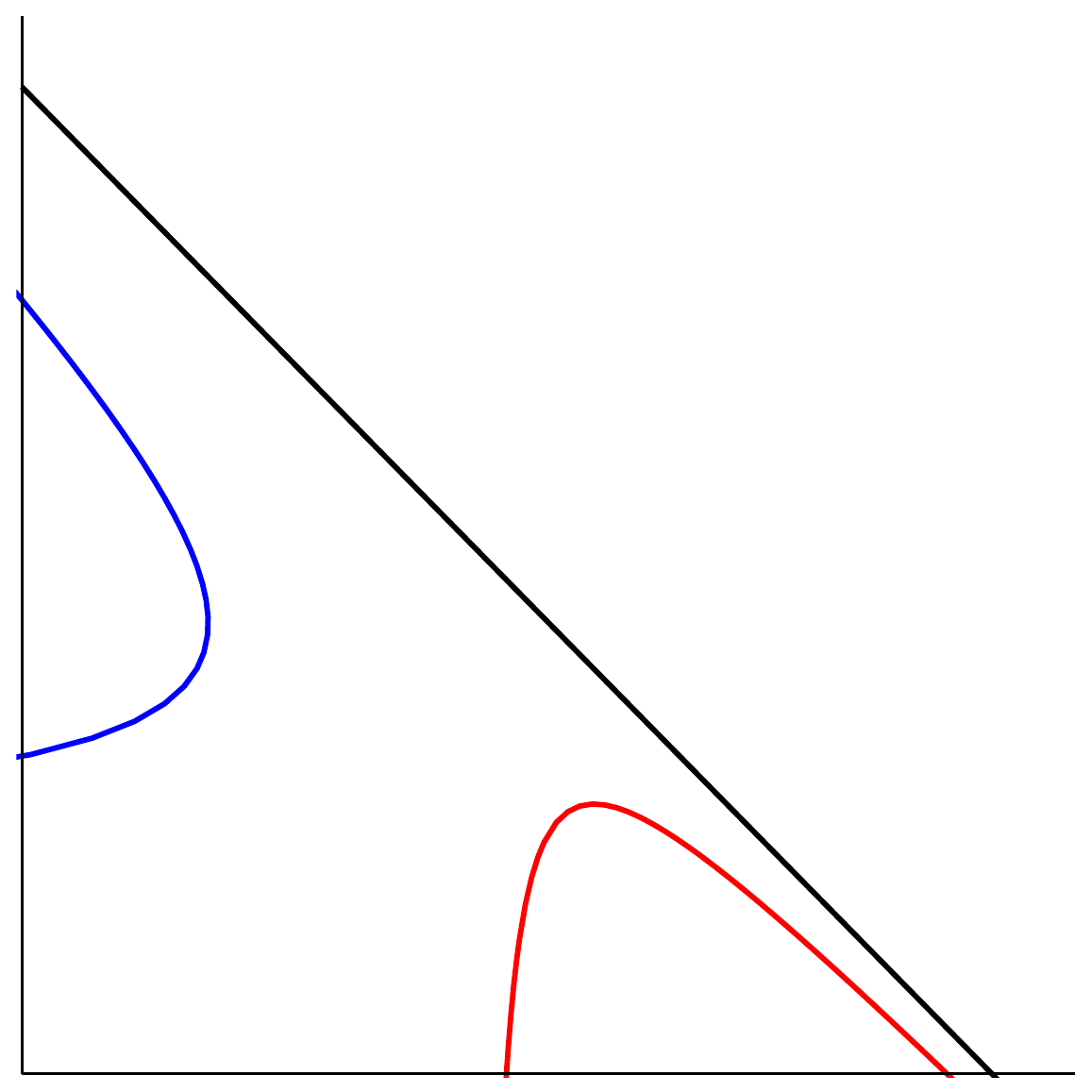}}}
\put(3.6,0){\rotatebox{0}{\includegraphics[width=6.5cm,height=6cm]{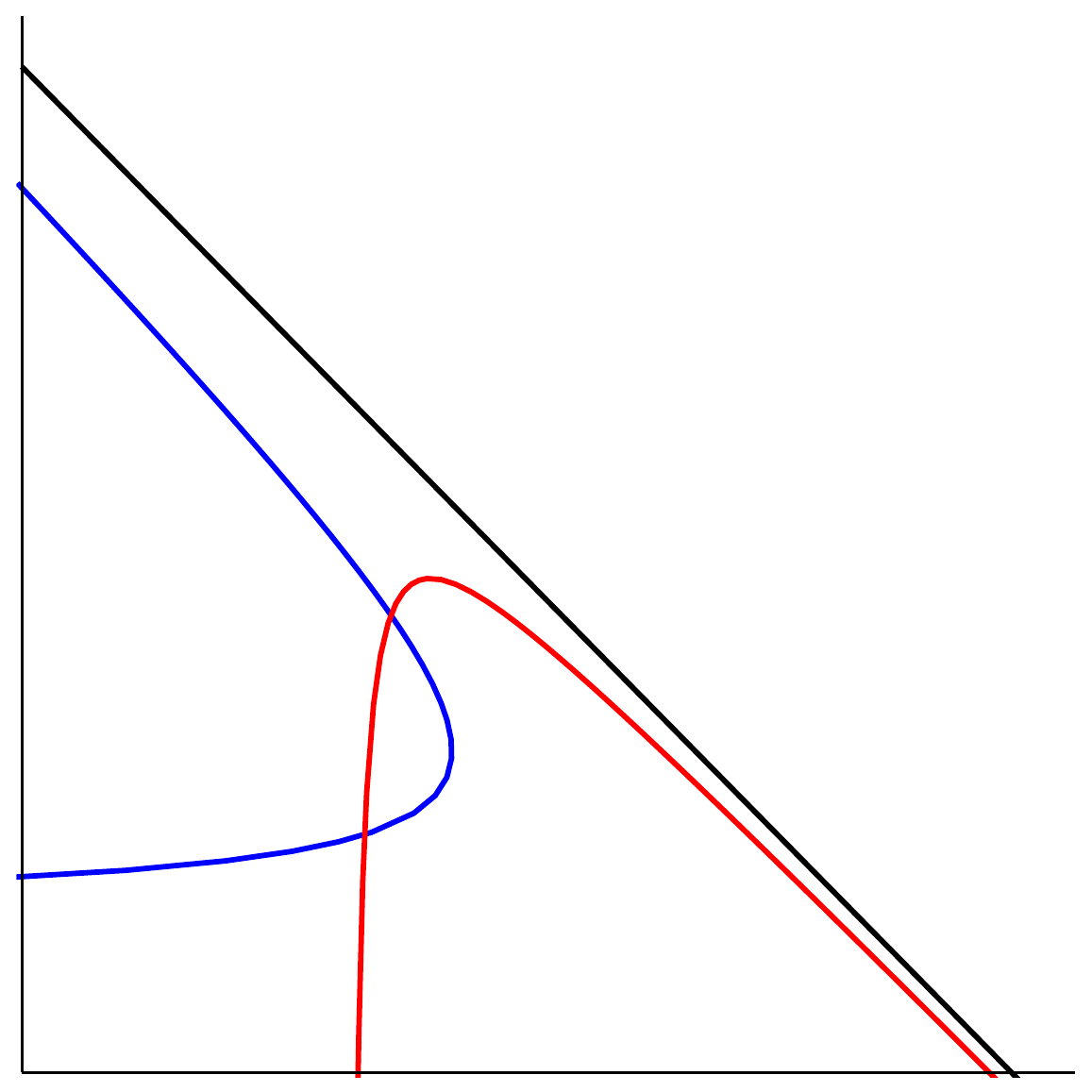}}}
\put(-0.9,5.7){{\sc $(a)$}}
\put(-3.3,5.7){{\sc  $x_2$}}
\put(-4.4,5.4){{\sc  $\frac{D}{D_2}S_{in}$}}
\put(-3.9,4.3){{\sc  $x_2^2$}}
\put(-3.9,1.7){{\sc  $x_2^1$}}
\put(-2.3,2.6){{\sc {\color{blue}$\gamma_1$}}}
\put(-0.4,2.7){{\sc $\delta$}}
\put(-0.7,1.45){{\sc {\color{red}$\gamma_2$}}}
\put(-0.75,-0.2){{\sc  $x_1^1$}}
\put(1.8,-0.2){{\sc  $x_1^2$}}
\put(2.15,-0.25){{\sc  $\frac{D}{D_1}S_{in}$}}
\put(2.7,0.2){{\sc  $x_1$}}
\put(6.5,5.7){{\sc $(b)$}}
\put(3.8,5.7){{\sc  $x_2$}}
\put(2.8,5.55){{\sc  $\frac{D}{D_2}S_{in}$}}
\put(3.3,4.8){{\sc  $x_2^2$}}
\put(3.3,1.1){{\sc  $x_2^1$}}
\put(5.15,3.6){{\sc {\color{blue}$\gamma_1$}}}
\put(6.15,3.5){{\sc $\delta$}}
\put(5.4,0.8){{\sc {\color{red}$\gamma_2$}}}
\put(5.6,-0.2){{\sc  $x_1^1$}}
\put(9.2,-0.2){{\sc  $x_1^2$}}
\put(9.5,-0.25){{\sc  $\frac{D}{D_1}S_{in}$}}
\put(9.9,0.2){{\sc  $x_1$}}
\end{picture}
\end{center}
\caption{Multiplicity of positive equilibria of system \cref{ModelMutualism}.
Intersections of the curves $\gamma_1$ and $\gamma_2$ inside the invariant set $\mathring M$: $(a)$ no positive equilibrium, $(b)$ an even number of positive equilibria.}\label{Chap5-FigNull_MulEet}
\end{figure}
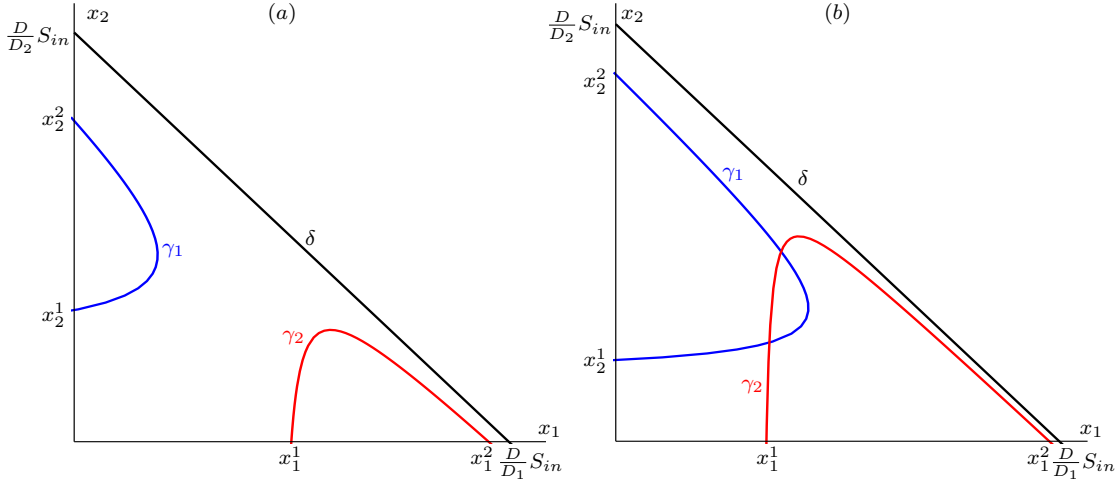 

\Cref{AppendixA} shows that, for $i,j\in\{1,2\}$, with $i\neq j$, the function $F_i(x_j)$ has a unique extremum in the interval $[x_j^1, x_j^2]$ when the specific growth functions \cref{SpeciFunc} are considered.
As a consequence of the preceding lemmas, we obtain the following result, which provides a condition for the existence of positive equilibria.
\begin{proposition}                                  \label{PropExEet}
System \cref{ModelMutualism} has a positive equilibrium
$\mathcal{E}^*=(S^*,x_1^*,x_2^*)$ if and only if the curves $\gamma_1$ and $\gamma_2$ intersect at a positive point $(x_1^*,x_2^*)$ satisfying
\begin{equation}                                    \label{EquExisEet}
x_1 = F_1(x_2), \quad x_2 = F_2(x_1),
\end{equation}
with $S^*$ given by \cref{ExpSetoil}.
Moreover, if a positive equilibrium exists, it is not necessarily unique; generically, system \cref{ModelMutualism} may exhibit an even number of positive equilibria (see \cref{Chap5-FigNull_MulEet}).
\end{proposition}
\begin{proof}
A positive equilibrium $\mathcal{E}^*$ of system \cref{ModelMutualism} corresponds to a solution of \cref{EqsExiEet} in $\mathring{M}$.
By \cref{LemFi}, such a solution exists if and only if the curves $\gamma_1$ and $\gamma_2$ intersect at a positive point $(x_1^*,x_2^*)$ satisfying \cref{EquExisEet}, with $S^*$ given by \cref{ExpSetoil}.
\end{proof}

We now analyze the local stability of all equilibria of system~\cref{ModelMutualism}, using the acronym LES to denote locally exponentially stable equilibria.
For the washout equilibrium $\mathcal{E}_0$, stability is determined directly from the eigenvalues of the Jacobian matrix of~\cref{ModelMutualism}, which in this case reduce to its diagonal entries evaluated at $\mathcal{E}_0$.
For the coexistence equilibrium $\mathcal{E}^*$, stability is established using the Routh--Hurwitz criterion.
Using the notation introduced in \cref{ExprEFGH}, the Jacobian matrix of system \cref{ModelMutualism} at an equilibrium $(S, x_1, x_2)$ takes the form
\begin{equation*}
J =
\begin{pmatrix}
-D - Ex_1 - Fx_2 & -f_1(S,x_2) - Hx_2 & - f_2(S,x_1) - Gx_1  \\
            Ex_1 & f_1(S,x_2) - D_1   & Gx_1 \\
            Fx_2 & Hx_2               & f_2(S,x_1) - D_2 \\
\end{pmatrix}.
\end{equation*}
The following result establishes the local stability of the washout equilibrium $\mathcal{E}_0$.
\begin{proposition}                              \label{PropStabE0}
Under \cref{hyp1,hyp2,hyp3,hyp4}, the washout equilibrium $\mathcal{E}_0$ is LES.
\end{proposition}
\begin{proof}
At $\mathcal{E}_0=(S_{in},0,0)$, the Jacobian matrix is diagonal, and its characteristic polynomial is
\begin{equation*}
P(\lambda) = (-D - \lambda)(-D_1 - \lambda)(-D_2 - \lambda).
\end{equation*}
Since all eigenvalues are negative, the equilibrium $\mathcal{E}_0$ is always LES.
\end{proof}

We now study the local stability of the positive equilibria.
Let $J^*$ denote the Jacobian matrix of system \cref{ModelMutualism} evaluated at a positive equilibrium
$\mathcal{E}^* = (S^*,x_1^*,x_2^*)$.
It is given by
\begin{equation*}
J^* =
\begin{pmatrix}
-D - E x_1^* - F x_2^* & -D_1 - H x_2^* & -G x_1^* - D_2 \\
E x_1^*                & 0            & G x_1^*\\
F x_2^*                & H x_2^*      & 0
\end{pmatrix},
\end{equation*}
where the quantities $E$, $F$, $G$, and $H$ defined in \cref{ExprEFGH} are evaluated at $(S^*,x_1^*,x_2^*)$.

The characteristic polynomial of $J^*$ has the form
$$
P(\lambda) = \lambda^3 + c_1 \lambda^2 + c_2 \lambda + c_3,
$$
whose coefficients are
\begin{equation}                               \label{Expr-CRH}
\begin{aligned}
c_1 &= D + E x_1^* + F x_2^*, \\
c_2 &= D_1 E x_1^* + D_2 F x_2^* + (F G + E H - G H) x_1^* x_2^*, \\
c_3 &= (D_1 F G + D_2 E H - D G H) x_1^* x_2^*.
\end{aligned}
\end{equation}
 
Since $c_1>0$, the Routh--Hurwitz criterion implies that the positive equilibrium $\mathcal{E}^*$ is LES if and only if
$$
c_3 > 0 \quad \mbox{and} \quad c_4 := c_1 c_2 - c_3 > 0.
$$
Note that, if $c_1>0$ and $c_3>0$, then $c_4>0$ implies $c_2>0$.

The next result establishes a precise relationship between the sign of the coefficient $c_3$ and the geometric configuration of the curves
$\gamma_1$ and $\gamma_2$, defined respectively by $x_2 \mapsto F_1(x_2)=x_1$ and $x_1 \mapsto F_2(x_1)=x_2$.
More specifically, we relate the determinant of the Jacobian matrix $\det(J^*)$ at the positive equilibrium $\mathcal{E}^*=(S^*,x_1^*,x_2^*)$
to the quantity $F'_1(x_2^*)F'_2(x_1^*)-1$, which measures the relative slopes of the two curves at their intersection point.
We show that the condition $c_3>0$ is equivalent to $F'_1(x_2^*)F'_2(x_1^*)-1<0$.
Geometrically, this means that, at the equilibrium point, the tangent to $\gamma_1$ lies below the tangent to $\gamma_2$.
This relative position plays a key role in determining the local stability of the coexistence equilibrium.

\begin{lemma}  
Let $\mathcal{E}^* = (S^*,x_1^*,x_2^*)$ be a positive equilibrium of system \cref{ModelMutualism}. Then,
\begin{equation}                                \label{Expc3}
c_3 \;=\; -\det(J^*) \;=\; -\frac{D_1 D_2 E F}{D}\bigl(F'_1(x_2^*)F'_2(x_1^*) - 1\bigr)x_1^* x_2^*.
\end{equation}
Consequently,
$$
c_3 > 0 \quad \Longleftrightarrow \quad F'_1(x_2^*)F'_2(x_1^*) - 1 < 0.
$$
\end{lemma}
\begin{proof}
Using the expressions of the derivatives of $F_1$ and $F_2$ given in \cref{ExpF1prim,ExpF2prim}, evaluated at $(x_1^*,x_2^*)$, a straightforward computation yields
\begin{equation}                            \label{RelF12prim}
F'_1(x_2^*)F'_2(x_1^*) - 1  \;=\; \frac{-D\bigl(D_1 F G + D_2 E H - D G H\bigr)}{D_1 D_2 E F}.
\end{equation}
On the other hand, combining the coefficient $c_3$ in the characteristic polynomial \cref{Expr-CRH} with \cref{RelF12prim} immediately gives
$$
c_3 =-\frac{D_1 D_2 E F}{D}\bigl(F'_1(x_2^*)F'_2(x_1^*) - 1\bigr)x_1^* x_2^*,
$$
which proves \cref{Expc3}.
\end{proof}
In the following proposition, we provide necessary and sufficient conditions for the local stability of a positive equilibrium of model \cref{ModelMutualism}.
\begin{proposition}                          \label{PropStabEe}
Assume that \cref{hyp1,hyp2,hyp3,hyp4} and condition \cref{CondExis_xji} hold.
A positive equilibrium $\mathcal{E}^* = (S^*,x_1^*,x_2^*)$ is LES if and only if 
$$
F'_1(x_2^*)\,F'_2(x_1^*) < 1 \quad\text{and}\quad c_4 > 0,
$$
where $c_4$ is given by
\begin{equation}                                  \label{Expc4}
\begin{aligned}
c_4 \;:=\;& D_1 E^2 {x_1^*}^2 + D_2 F^2 {x_2^*}^2 + D D_1 E x_1^* + D D_2 F x_2^* \\
       &+ \bigl((D_1 + D_2) E F + (D - D_1) F G + (D - D_2) E H \bigr) x_1^* x_2^* \\
       &+ (F G + E H - G H)\,\bigl(E {x_1^*}^2 x_2^* + F x_1^* {x_2^*}^2 \bigr).
\end{aligned}
\end{equation}
\end{proposition}
\noindent
 
Although it is not possible to prove analytically that the condition $c_4>0$ always holds, we have succeeded in identifying parameter sets for which $c_4$ may change sign. As will be illustrated in the next section, varying one of the control parameters can lead to $c_4$ becoming negative, which corresponds to a loss of stability of the positive equilibrium through a Hopf bifurcation.
\section{Hopf bifurcation}                      \label{SectionHopfBif}
Having established the conditions for the existence and local stability of positive equilibria, we now investigate the onset of oscillatory dynamics in system \cref{ModelMutualism}. In particular, we focus on the occurrence of Hopf bifurcations, where a stable equilibrium loses stability and a limit cycle emerges, as one of the control parameters is varied.

To illustrate these phenomena, we consider the specific growth functions satisfying \cref{hyp1,hyp2,hyp3,hyp4}:
\begin{equation}                                  \label{SpeciFunc}
f_1(S,x_2)= \frac{m_1 S}{K_1+S}\frac{x_2}{L_1+x_2},
\qquad
f_2(S,x_1)= \frac{m_2 S}{K_2+S}\frac{x_1}{L_2+x_1},
\end{equation}
where $m_1,m_2$ denote the maximum growth rates, $K_1,K_2$ are the Michaelis–Menten constants associated with the nutrient uptake, and $L_1,L_2$ are positive constants describing the mutualistic interaction. 
The biological parameter values are given in \cref{Tab-Allpar} (line 1). 
In the following analysis, we fix the dilution rate at $D=0.2$ and vary the input concentration $S_{in}$ to examine the bifurcation structure of the system.

In this parameter regime, the system undergoes a rich bifurcation scenario as the input concentration $S_{in}$ varies. 
In particular, two positive equilibria, denoted by $\mathcal{E}_1^*$ and $\mathcal{E}_2^*$, emerge through a Limit Point (LP) bifurcation at
$S_{in}=\sigma_1 \approx 2.850$ (see \cref{FigVP}(a)). 
The equilibrium $\mathcal{E}_2^*$ is unstable since $c_3<0$, which occurs precisely when the tangent to $\gamma_1$ at $(x_1^*,x_2^*)$ lies above that of $\gamma_2$ at the same point.
By contrast, $\mathcal{E}_1^*$ also arises unstable, but now with $c_3>0$ and $c_4<0$ for
$\sigma_1 < S_{in} < \sigma_6 \approx 3.238$ (see \cref{FigVP}(b)).
 
The critical value $\sigma_6$, defined by the condition $c_4(S_{in})=0$, marks a change in the stability of $\mathcal{E}_1^*$ without collision with any other equilibrium.
For $S_{in}>\sigma_6$, the equilibrium $\mathcal{E}_1^*$ becomes LES, as both conditions $c_3>0$ and $c_4>0$ are satisfied. The stabilization of $\mathcal{E}^*$ at the critical value $\sigma_6$ as $S_{in}$ increases naturally raises the question whether $\sigma_6$ corresponds to a Hopf bifurcation.
\begin{figure}[!ht]
\setlength{\unitlength}{1.0cm}
\begin{center}
\begin{picture}(6.3,5.3)(0,0)
\put(-3.3,0){\rotatebox{0}{\includegraphics[width=4cm,height=5cm]{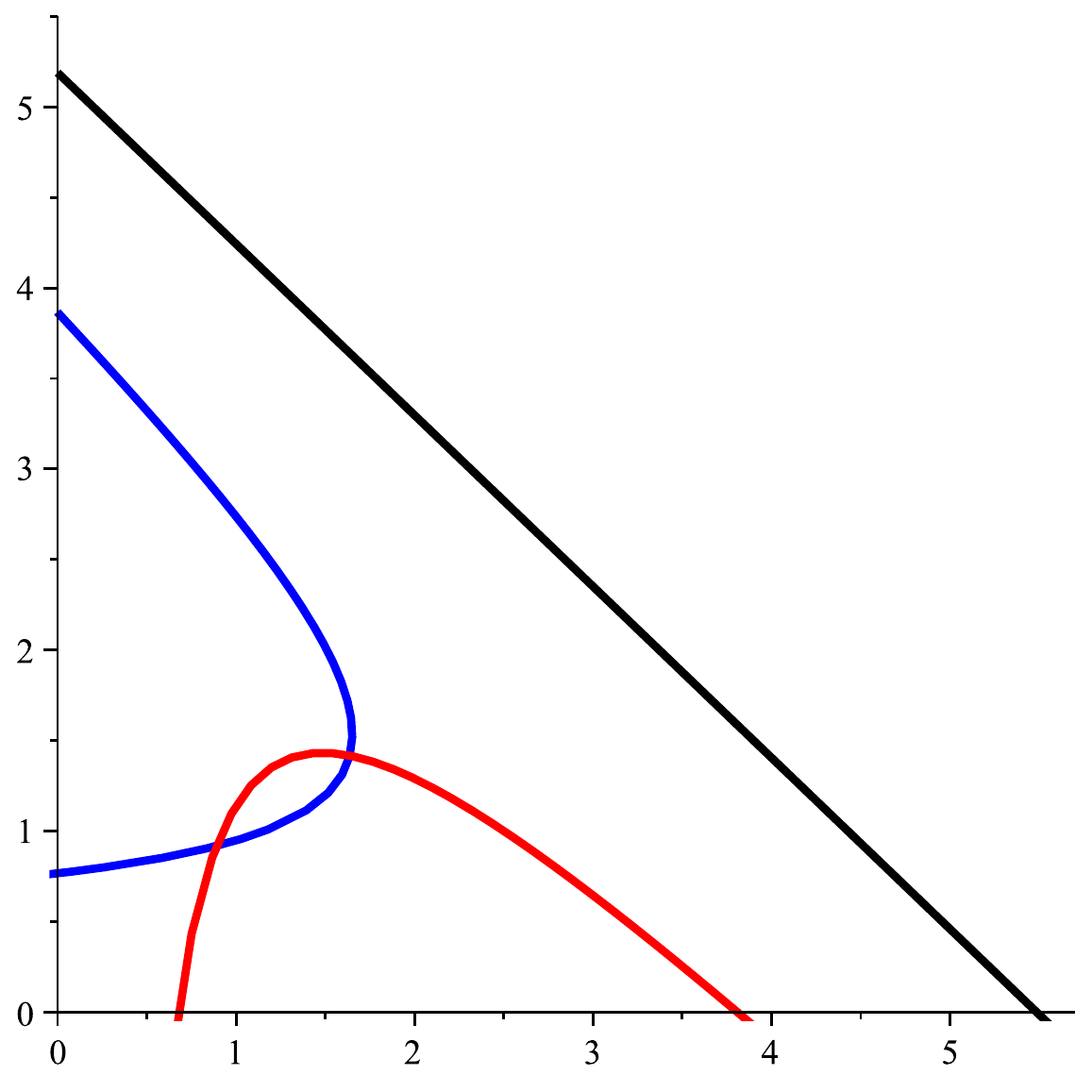}}}
\put(1.15,0){\rotatebox{0}{\includegraphics[width=4cm,height=5cm]{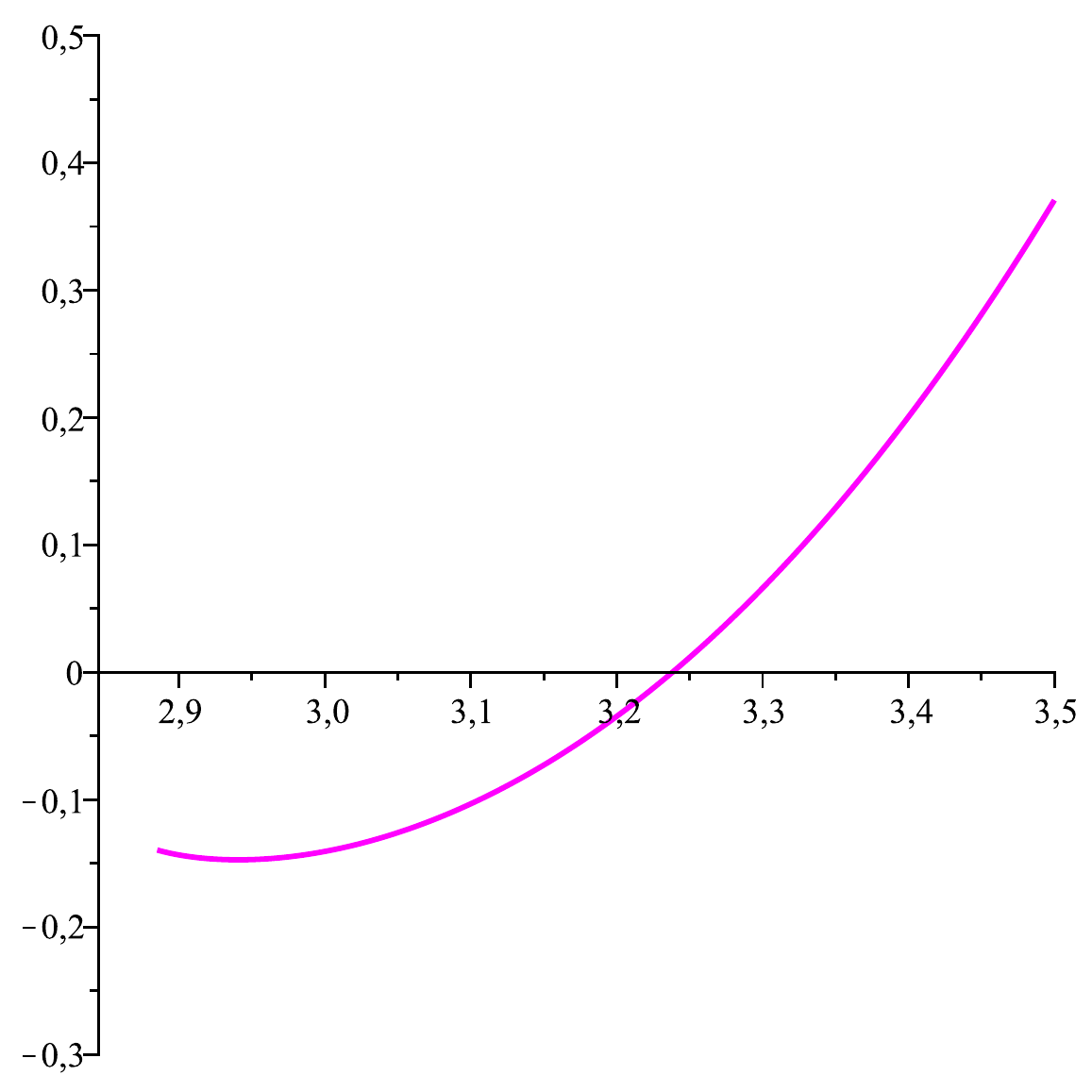}}}
\put(5.8,0){\rotatebox{0}{\includegraphics[width=4cm,height=5cm]{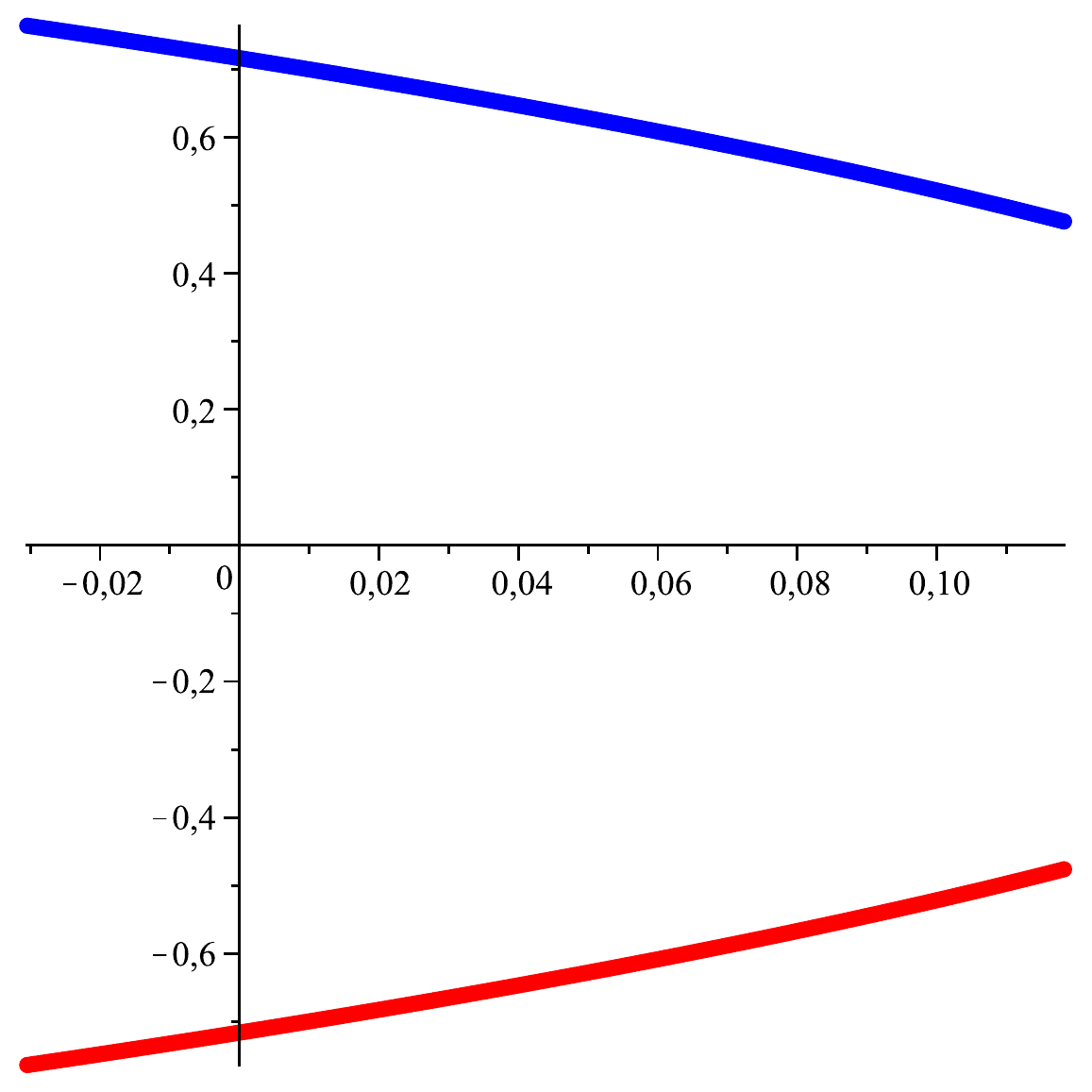}}}
\put(-1.8,5){{\sc (a)}}
\put(-1.4,2.8){{\sc $\delta$}}
\put(-1.95,1.6){{\sc {\color{blue} $\mathcal{E}_2^*$}}}
\put(-2.08,1.46){{\sc {\color{blue} $\bullet$}}}
\put(-2.4,0.9){{\sc {\color{blue} $\mathcal{E}_1^*$}}}
\put(-2.58,1.05){{\sc {\color{blue} $\bullet$}}}
\put(-2.98,0.5){{\sc {\color{red} $\mathcal{E}_0$}}}
\put(-3.15,0.3){{\sc {\color{red} $\bullet$}}}
\put(-2.35,2.6){{\sc {\color{blue} $\gamma_1$}}}
\put(-1.1,1){{\sc {\color{red} $\gamma_2$}}}
\put(0.55,0.45){{\sc $x_1$}}
\put(-3,4.75){{\sc $x_2$}}
\put(2.8,5){{\sc (b)}}
\put(4.65,4){{\sc {\color{magenta} $c_4$}}}
\put(3.5,1.6){{\sc $\sigma_6$}}
\put(3.55,1.85){{\sc  $\dagger$}}
\put(5.1,1.85){{\sc  $S_{in}$}}
\put(7.6,5){{\sc (c)}}	
\put(6.2,5){\sc $\nu (S_{in})$}	
\put(9.85,2.5){\sc $\mu (S_{in})$}
\put(7.2,0.6){\sc {\color{red} $\lambda_-$}}
\put(6,0.35){\sc {\color{red} \vector(4,1){0.5}}}
\put(7.2,4.3){\sc {\color{blue} $\lambda_+$}}
\put(6,4.65){\sc {\color{blue} \vector(4,-1){0.5}}}
\end{picture}
\end{center}
\caption{(a) Existence and instability of two positive equilibria for $\sigma_1 < S_{in} < \sigma_6$ with $(S_{in},D)=(3,0.2)$;
(b) sign change of $c_4(S_{in})$ for $S_{in} \geq \sigma_6$ with $D=0.2$;
(c) evolution of a pair of complex-conjugate eigenvalues as $S_{in}$ decreases. Unstable equilibria are indicated in blue.}
\label{FigVP}
\end{figure}
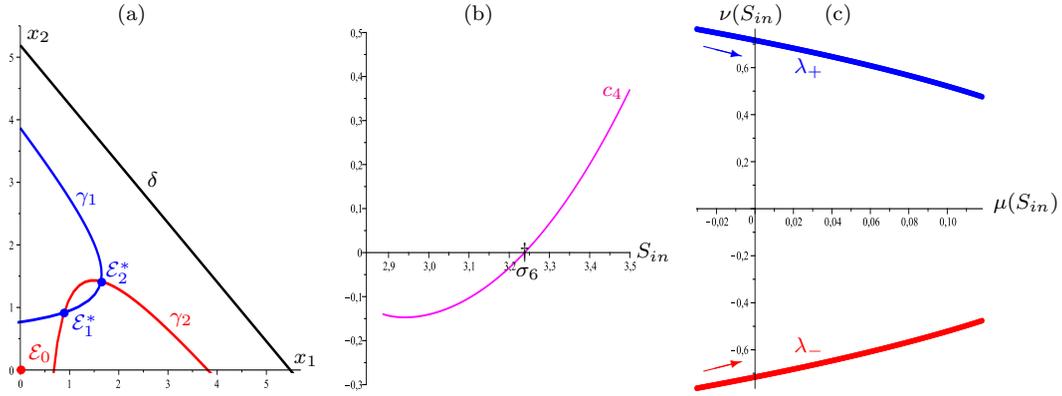

To substantiate the occurrence of a Hopf bifurcation at $\sigma_6$, we numerically computed the eigenvalues of the Jacobian matrix $J^*$ of system \cref{ModelMutualism} at the equilibrium $\mathcal{E}_1^*$ by evaluating the roots of its characteristic polynomial as the input concentration $S_{in}$ varies. 
For all $S_{in}>\sigma_1$, the spectrum of $J^*$ consists of one negative real eigenvalue and a pair of complex-conjugate eigenvalues
$$
\lambda_{\pm}=\mu(S_{in}) \pm i \nu(S_{in}),
$$
where $\mu(S_{in})>0$ for $S_{in}\in[\sigma_1,\sigma_6)$ and $\mu(S_{in})<0$ for $S_{in}>\sigma_6$. 
At the critical value $S_{in}=\sigma_6$, the pair $\lambda_{\pm}$ is purely imaginary, that is, $\mu(\sigma_6)=0$ and $\nu(\sigma_6)\neq 0$.
Moreover, the transversality condition
\[
\frac{\partial \mu}{\partial S_{in}}(\sigma_6) > 0
\]
is verified numerically, ensuring that the eigenvalues cross the imaginary axis transversally.
As illustrated in \cref{FigVP}(c), decreasing $S_{in}$ causes the pair $\lambda_{\pm}$ to cross the imaginary axis at $S_{in}=\sigma_6$, from the left to the right half-plane.

Consequently, a Hopf bifurcation occurs at $S_{in}=\sigma_6$. When $S_{in}$ decreases below this critical value, the equilibrium $\mathcal{E}_1^*$ becomes unstable and a stable limit cycle appears. The bifurcation diagram in \cref{Fig-DB-D02} confirms this scenario, showing the emergence of stable limit cycles and sustained oscillations in agreement with the bifurcation analysis.
\section{Case without mortality}         \label{SubSec-EquilibSanMort}
In this section, we consider a particular case of system \cref{ModelMutualism} in which species mortality is neglected, that is, we assume $D_1 = D_2 = D$.
Under this assumption, system \cref{ModelMutualism} reduces to
\begin{equation}                                \label{ModelDiEgaux}
\left\{\begin{array}{lll}
\dot{S}   &=& D(S_{in}-S) - f_1(S,x_2)x_1 - f_2(S,x_1)x_2,\\
\dot{x}_1 &=& (f_1(S,x_2)-D)x_1, \\ 
\dot{x}_2 &=& (f_2(S,x_1)-D)x_2.
\end{array}\right.
\end{equation}
This simplified system has been previously investigated by El Hajji et al.
\cite{ElhajjiIJB2018,ElhajjiJBD2009}, where the existence and local stability of equilibria are characterized in terms of the maximum values of the density-dependent growth rates and the dilution rate $D$.
Here, we adopt a different perspective based on the intersection of the isoclines $\gamma_i$ and their relative position.
This geometric approach provides a transparent and intuitive criterion for both the existence of positive equilibria and their local stability.

Using the general results established for distinct removal rates ($D_i \neq D$), we now specialize our analysis to the case $D_i = D$ for $i=1,2$, and obtain the following result.               
\begin{proposition}                               \label{PropExisStab}    
Assume that \cref{hyp1,hyp2,hyp3,hyp4} and condition \cref{CondExis_xji} hold. 
Then, for system \cref{ModelDiEgaux}, all equilibria together with the necessary and sufficient conditions for their existence and local exponential stability are summarized in \cref{Tab-ExiStabDegaux}.  
\end{proposition}                             
\begin{table}[ht] 
\caption{Equilibria of system \cref{ModelDiEgaux}, their components, and the corresponding conditions for existence and local stability. The functions $F_1$ and $F_2$, as well as the associated isoclines $\gamma_1$ and $\gamma_2$, are defined in \cref{LemFi} for $D_i = D$.}
\label{Tab-ExiStabDegaux} 
\begin{center}
\begin{tabular}{ @{\hspace{1mm}}l@{\hspace{2mm}} @{\hspace{2mm}}l@{\hspace{2mm}}  @{\hspace{2mm}}l@{\hspace{2mm}}   @{\hspace{2mm}}l@{\hspace{2mm}}  @{\hspace{1mm}}l@{\hspace{1mm}} }
& Components             & Existence condition & Stability condition \\ \hline 
$\mathcal{E}_0$ 
& $S=S_{in}$, $x_1=x_2=0$ 
& Always              
& Always        
\\ \hline  
$\mathcal{E}^*$ 
& \hspace{-0.2cm}
\begin{tabular}{l}     
  $S^*=S_{in}-x_1^*-x_2^*$, \\ 
  $x_1^*=F_1(x_2^*)$, \ $x_2^*=F_2(x_1^*)$ 
\end{tabular} 
& $\gamma_1$ and $\gamma_2$ intersect  
& $F_1'(x_2^*)F_2'(x_1^*) < 1$          
\\
\end{tabular}
\end{center}
\end{table}
\begin{proof}
The washout equilibrium $\mathcal{E}_0 = (S_{in},0,0)$ always exists. 
The coordinates of a positive equilibrium $\mathcal{E}^*$, as well as its existence condition, follow directly from \cref{PropExEet} by setting $D_i = D$ in the functions $F_i$ defined in \cref{LemFi}. 
For stability, \cref{PropStabE0} shows that, when $D_i = D$, the Jacobian matrix at $\mathcal{E}_0$ has a triple eigenvalue equal to $-D$, so $\mathcal{E}_0$ is always LES. 
For $\mathcal{E}^*$, \cref{PropStabEe} gives the stability condition reported in \cref{Tab-ExiStabDegaux}. In particular, we have
$$
c_3 > 0 \quad \Longleftrightarrow \quad F_1'(x_2^*)F_2'(x_1^*) < 1,
$$
since the coefficient $c_4$ defined in \cref{Expc4} is positive when $D_i = D$. Indeed, in this case all terms in the expression of $c_4$ are positive. Moreover, using relation \cref{RelF12prim}, we obtain
\[
F G + E H - G H \;=\; \big(1-F_1'(x_2^*)F_2'(x_1^*)\big) E F,
\]
which is positive whenever $c_3 > 0$, confirming the stated stability condition.
\end{proof}

Since the washout equilibrium $\mathcal{E}_0$ always exists and is LES, its global stability can occur only in the absence of any positive equilibrium. Indeed, positive equilibria emerge through saddle–node bifurcations, where stable nodes and saddle points alternate. 

To analyze the global stability of $\mathcal{E}_0$, the three-dimensional system \cref{ModelDiEgaux} can be reduced to a planar system. Introducing the total mass density 
\[
z = S + x_1 + x_2,
\] 
and expressing the dynamics in terms of $(z, x_1, x_2)$, the system can be rewritten as
\begin{equation}                          \label{ModelDiEgaux_Varz}
\left\{
\begin{aligned}
\dot{z}   &= D(S_{in}-z),\\
\dot{x}_1 &= \bigl(f_1(z-x_1-x_2,x_2)-D\bigr)x_1, \\ 
\dot{x}_2 &= \bigl(f_2(z-x_1-x_2,x_1)-D\bigr)x_2.
\end{aligned}
\right.
\end{equation}
Let $(z(t), x_1(t), x_2(t))$ be a solution of \cref{ModelDiEgaux_Varz}.  
From the first equation, we immediately obtain
\[
z(t) = S_{in} + \bigl(z(0) - S_{in}\bigr)e^{-Dt}.
\]  
Hence, system \cref{ModelDiEgaux_Varz} is equivalent to the following non-autonomous system of two differential equations:
\begin{equation}                          \label{ModelDiEgaux_Asym}
\left\{
\begin{aligned}
\dot{x}_1 &= \bigl(f_1(S_{in} + (z(0)-S_{in})e^{-Dt} - x_1 - x_2,\, x_2) - D\bigr)x_1, \\[1ex] 
\dot{x}_2 &= \bigl(f_2(S_{in} + (z(0)-S_{in})e^{-Dt} - x_1 - x_2,\, x_1) - D\bigr)x_2.
\end{aligned}
\right.
\end{equation}
System \cref{ModelDiEgaux_Asym} is asymptotically autonomous, converging as $t \to +\infty$ to the autonomous system
\begin{equation}                        \label{ModelDiEqualModRed}
\left\{
\begin{aligned}
\dot{x}_1 &= \bigl(f_1(S_{in}-x_1-x_2,\,x_2)-D\bigr)x_1, \\[1ex] 
\dot{x}_2 &= \bigl(f_2(S_{in}-x_1-x_2,\,x_1)-D\bigr)x_2.
\end{aligned}
\right.
\end{equation}
By applying the results of Thieme \cite{ThiemeJMB1992}, which were used in a manner similar to \cite{ElhajjiIJB2018,ElhajjiJBD2009}, it follows that the asymptotic behavior of the full three-dimensional system \cref{ModelDiEgaux_Varz} coincides with that of the two-dimensional reduced system \cref{ModelDiEqualModRed}.
\begin{proposition} 
The washout equilibrium $\mathcal{E}_0$ of system \cref{ModelDiEgaux} is globally asymptotically stable whenever it is the unique equilibrium.
\end{proposition}
\begin{figure}[!ht]
\setlength{\unitlength}{1.0cm}
\begin{center}
\begin{picture}(6.5,5)(0,0)
\put(0,0){\rotatebox{0}{\includegraphics[width=8cm,height=8cm]{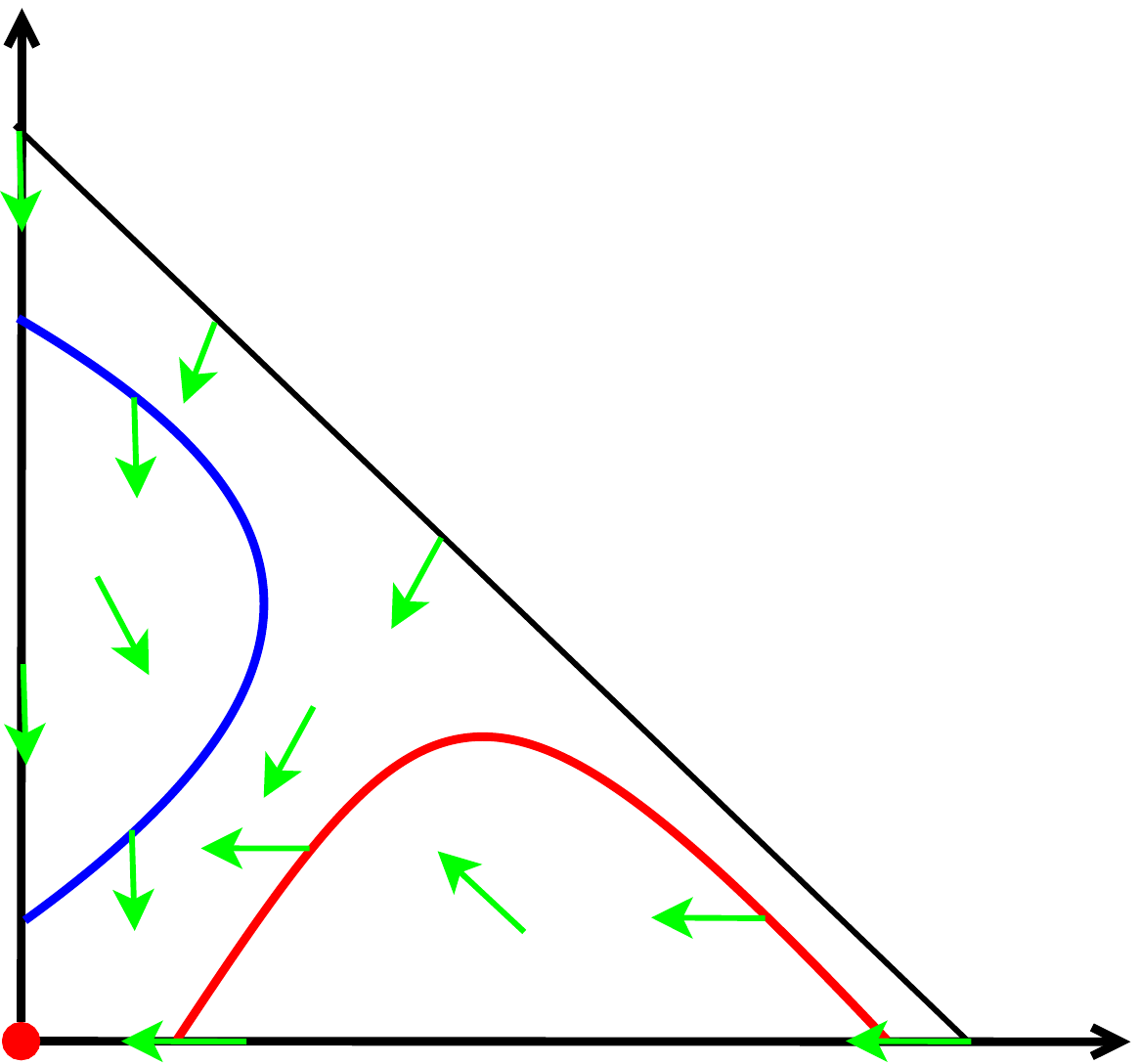}}}
\put(0.2,5){{\sc  $x_2$}}
\put(1.65,2.5){{\sc {\color{blue}$\gamma_1$}}}
\put(2.5,1.55){{\sc {\color{red}$\gamma_2$}}}
\put(0.1,-0.25){{\sc  {\color{red}$E_0$}}}
\put(2.6,2.8){{\sc $\delta$}}
\put(2.4,0.6){{\sc I}}
\put(2,2){{\sc II}}
\put(1.1,2.1){{\sc III}}
\put(7.5,0.2){{\sc  $x_1$}}
\end{picture}
\end{center}
\caption{Phase-plane representation of the reduced system \cref{ModelDiEqualModRed} illustrating the global stability of the washout equilibrium $E_0$, in the case where the nullclines $\gamma_1$ and $\gamma_2$ do not intersect.}\label{FigIsoE0GAS}
\end{figure} 
\begin{proof}
To establish global asymptotic stability, we perform a phase-plane analysis, following the approach of \cite[Proposition 7]{FekihMB2017} and \cite[Section 2.1.2.3]{HarmandBook2017}. For completeness, the details are given below.

The washout equilibrium $\mathcal{E}_0=(S_{in},0,0)$ of system \cref{ModelDiEgaux} corresponds to $E_0=(0,0)$ in the reduced system \cref{ModelDiEqualModRed}.  
When $E_0$ is the unique equilibrium, the nullclines $\gamma_1$ and $\gamma_2$ do not intersect (see \cref{FigIsoE0GAS}). They partition the interior of the set $M$ (defined in \cref{SetM} under $D_i = D$) into three closed and bounded regions, within which the reduced system \cref{ModelDiEqualModRed} is monotone. These regions are characterized by the signs of the derivatives as follows:
$$
I:  ~ \dot{x}_1<0, ~ \dot{x}_2>0, \quad
II: ~ \dot{x}_1<0, ~ \dot{x}_2<0, \quad
III: ~ \dot{x}_1>0, ~ \dot{x}_2<0.
$$
Consider a trajectory $(x_1(t),x_2(t))$ with initial condition $(x_1(0),x_2(0)) \in M$.  
Any trajectory entering regions $I$ or $III$ must leaves these regions and enters region $II$.  
Once in region $II$, both $x_1(t)$ and $x_2(t)$ are decreasing, which guarantees the existence of the limits
$$ 
\lim_{t\to+\infty} x_1(t)=x_{1\infty}, \quad \lim_{t\to+\infty} x_2(t)=x_{2\infty}.
$$
Hence, the limiting point $(x_{1\infty},x_{2\infty})$ is an equilibrium of the reduced system \cref{ModelDiEqualModRed} and necessarily lies in the closure $\overline{II}$.  
Since $E_0$ is the only equilibrium in $\overline{II}$, it follows that
$$
E_0=(x_{1\infty},x_{2\infty})=(0,0).
$$
Therefore, every trajectory converges to $E_0$, proving its global asymptotic stability for the reduced system \cref{ModelDiEqualModRed}.  
By applying Thieme’s theorem \cite{ThiemeJMB1992}, as in \cite{ElhajjiIJB2018,ElhajjiJBD2009}, this result extends to the full three-dimensional system \cref{ModelDiEgaux}, establishing the global asymptotic stability of $\mathcal{E}_0$.
\end{proof}
In the following, we explore how the system dynamics respond to variations in one or more control parameters, focusing on their impact on equilibrium states and possible bifurcations.
\section{Operating diagrams}                            \label{Sec-OD}
Operating diagrams offer a clear and concise way to visualize and summarize the asymptotic behavior of the system as a function of key operating parameters, such as the dilution rate and the input substrate concentration, which are the most readily adjustable in chemostat experiments \cite{HarmandBook2017,SmithBook1995}.
While dilution rate and input substrate concentration are the main control parameters, other biological parameters—such as microbial attachment and detachment rates or interaction strengths—can also be experimentally manipulated, allowing biologists to influence the system dynamics in a controlled manner \cite{FekihJBS2025}.

In the literature, operating diagrams can be investigated using either numerical or theoretical approaches.
From a numerical standpoint, two main strategies are commonly employed. The first consists in solving the algebraic equations defining equilibria with nonnegative components. For each biologically relevant equilibrium, the roots of the associated characteristic polynomial are then computed in order to assess local stability, typically using MATLAB or other scientific computing platforms. This approach is highly flexible and can accommodate complex models involving many state variables and parameters.

The second numerical strategy relies on continuation and correction algorithms to track the boundaries separating the regions of the operating diagram. Dedicated software packages such as \textsc{MatCont}, CONTENT, AUTO, or XPPAUT have been developed for this purpose. Their main advantage lies in their ability to detect and characterize subtle bifurcations, including limit points of cycles, cusp, Bogdanov--Takens, or Bautin bifurcations.

Alternatively, operating diagrams can be studied through a theoretical approach, based on the analytical characterization of the boundaries between the different regions by deriving conditions for the existence and stability of equilibria as functions of the operating parameters \cite{AbdellatifMBE2016,DaoudMMNP2018,DellalDCDSB2021,DellalDCDSB2022,FekihSIADS2021,FekihMB2017,FekihJBS2025,HarmandBook2017,HmidhiBMB2025,MtarIJB2021,MtarDcdsB2022,NouaouraJMB2022,SariProcesses2022,Sari2024AIMS,SariNonLinDyn2021,SARISIADS2026,SariMB2017}. To illustrate such diagrams, specific growth functions are usually selected by fixing all biological parameters. 
Despite being referred to as a theoretical approach, this methodology typically requires symbolic or numerical computations using \textsc{MAPLE} or MATLAB to construct the operating diagram boundaries.

A major advantage of this approach is that this analytical framework ensures that all dynamic regions are identified and that no qualitative behavior is omitted. For example, Sari et al. \cite{NouaouraSiap2021,NouaouraDcDsB2021,NouaouraJMB2022,SariMB2017} showed that regions of coexistence associated with a stable limit cycle were not detected in a numerical operating diagram of a model describing the anaerobic mineralization of chlorophenol in a three-step microbial food web \cite{WadeJTB2016}. This highlights a well-known limitation of purely numerical approaches, which may fail to capture certain dynamical regimes due to discretization, continuation paths, or initialization choices.

However, the main limitation of this theoretical framework appears in complex models with many state variables, where deriving the boundary curves analytically may become intractable. For instance, Hmidhi et al. \cite{HmidhiBMB2025} analyzed an eight-dimensional model and identified up to 70 distinct dynamic regions in the operating diagrams.

In the following, we analyze the operating diagrams of system \cref{ModelMutualism} with respect to the main operating parameters, namely the substrate concentration in the feed, $S_{in}$, and the dilution rate, $D$. To this end, we consider the specific growth functions \cref{SpeciFunc}, which satisfy \cref{hyp1,hyp2,hyp3,hyp4}. The same construction procedure applies to any other growth functions that meet these general assumptions, ensuring that the qualitative structure of the operating diagrams is preserved.
\subsection{Case without mortality}              \label{SubSec-SanMort}
We start with the simplest scenario, in which mortality is neglected, i.e., $D_1 = D_2 = D$. 
In this setting, we construct one- and two-parameter bifurcation diagrams for model \cref{ModelDiEgaux}, extending the analysis beyond what was previously considered in \cite{ElhajjiIJB2018,ElhajjiJBD2009}. 
To this end, we employ the specific growth functions defined in \cref{SpeciFunc}, together with the biological parameter values listed in \cref{Tab-Allpar}, setting $a_1 = a_2 = 0$.
\begin{proposition}
For the specific growth functions $f_1$ and $f_2$ defined in \cref{SpeciFunc}, and for the set of biological parameter values listed in \cref{Tab-Allpar}, the existence and local stability properties of the equilibria $\mathcal{E}_0$, $\mathcal{E}_1^*$, and $\mathcal{E}_2^*$ of system \cref{ModelDiEgaux} are completely characterized within the regions $\mathcal{J}_0$ and $\mathcal{J}_1$ of the operating diagram displayed in \cref{FigDOSSMort}(a), as summarized in \cref{Tab-ExiStabDO}. 
The corresponding one-parameter bifurcation diagram with respect to $S_{in}$ for $D=0.2$ is shown in \cref{FigDOSSMort}(b).
\end{proposition}
\begin{table}[ht]
\caption{Existence and local stability of the equilibria of \cref{ModelDiEgaux} in the regions of the operating diagram in \cref{FigDOSSMort}(a). The symbol S (resp. U) indicates a LES (resp. unstable) equilibrium, while a blank entry indicates nonexistence.}\label{Tab-ExiStabDO}
\begin{center}
\begin{tabular}{ @{\hspace{1mm}}l@{\hspace{2mm}} @{\hspace{2mm}}l@{\hspace{2mm}}  @{\hspace{2mm}}l@{\hspace{2mm}}   @{\hspace{2mm}}l@{\hspace{2mm}}  @{\hspace{1mm}}l@{\hspace{1mm}} }
Region            & $\mathcal{E}_0$ & $\mathcal{E}_1^*$ & $\mathcal{E}_2^*$ \\ \hline 
$\mathcal{J}_0$   &     S           &                   &                    \\  
$\mathcal{J}_1$   &     S           &      S            &      U              \\
\end{tabular}
\end{center}
\end{table}
\begin{figure}[!ht]
\setlength{\unitlength}{1.0cm}
\begin{center}
\begin{picture}(6.3,7)(0,0)
\put(-5.5,-4){\rotatebox{0}{\includegraphics[width=9.5cm,height=17cm]{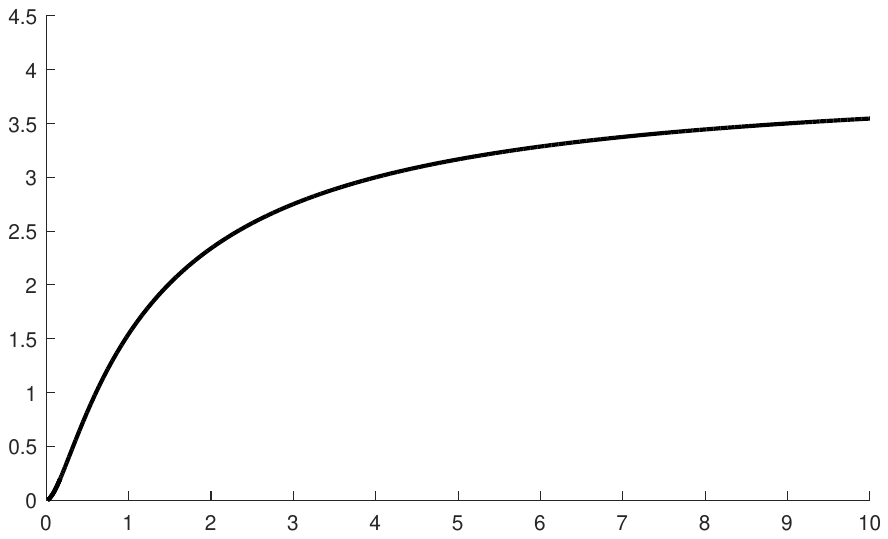}}}
\put(1.1,-3){\rotatebox{0}{\includegraphics[width=11cm,height=15cm]{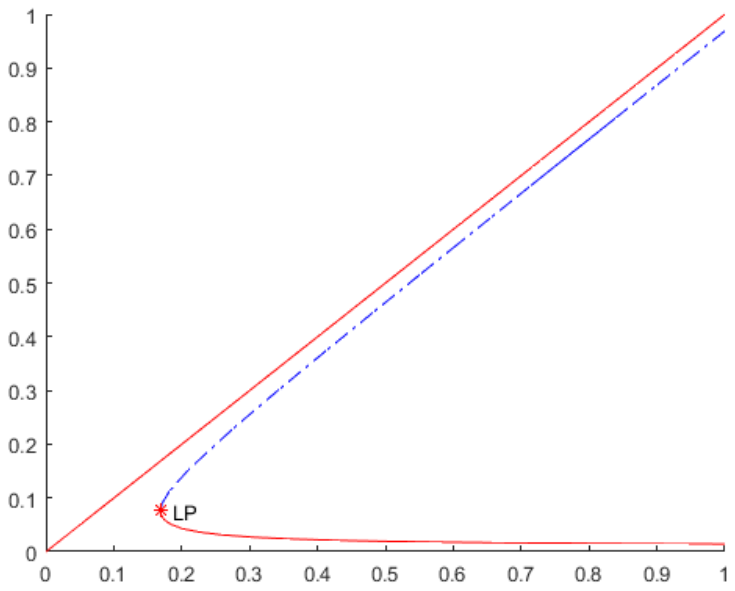}}}
\put(-0.5,7){{\sc (a)}}
\put(-3.7,6.85){{\sc $D$}}
\put(2.65,2.25){{\sc $S_{in}$}}
\put(-2,6.3){{\sc $\mathcal{J}_0$}}
\put(-0.8,4){{\sc $\mathcal{J}_1$}}
\put(2.65,5.9){{\sc $\Gamma_\sLP$}}
\put(6.6,7){{\sc (b)}}
\put(3.8,6.85){{\sc $S$}}
\put(9.8,2.25){{\sc $S_{in}$}}
\put(7.1,5.2){{\sc ${\color{red}\mathcal{E}_0}$}}
\put(7.1,2.55){{\sc ${\color{red}\mathcal{E}_1^*}$}}
\put(7.1,4.4){{\sc ${\color{blue}\mathcal{E}_2^*}$}}
\put(4.68,2.25){{\sc ${\color{red}\dagger}$}}
\put(4.55,2.05){{\sc $\sigma$}}
\end{picture}
\end{center}
\vspace{-2.2cm}
\caption{(a) Operating diagram of \cref{ModelDiEgaux} obtained with \textsc{MatCont}. (b) Corresponding one-parameter bifurcation diagram of the substrate concentration $S$  as a function of $S_{in}$ for $D=0.2$.}\label{FigDOSSMort}
\end{figure}
\begin{proof}
Positive equilibria of \cref{ModelDiEgaux} correspond to the intersection points of the nullclines $x_2 = F_1(x_1)$ and $x_1 = F_2(x_2),$
in the $(x_1,x_2)$-plane. 
A saddle–node bifurcation occurs when these two nullclines are tangent to each other. 
This tangency condition is characterized by
\begin{equation}                                   \label{Eqs3SNProof}
x_2 = F_1(x_1), \quad x_1 = F_2(x_2), \quad \text{and} \quad  F_1'(x_1)F_2'(x_2)=1,
\end{equation}
which is equivalent to the condition $c_3(x_1,x_2)=0$, according to the definition of $c_3$ given by \cref{Expc3}.
System \eqref{Eqs3SNProof} consists of three equations with four unknowns $(x_1,x_2,S_{in},D)$, and generically defines a curve 
$$
\Gamma_\sLP = \left\{(S_{in},D): \text{system \eqref{Eqs3SNProof} holds} \right\}
$$
in the $(S_{in},D)$-plane.
Crossing this curve in parameter space leads to the creation or annihilation of a pair of positive equilibria through a saddle–node bifurcation.
The one-parameter bifurcation diagram shown in \cref{FigDOSSMort}(b) illustrates the dynamical transition occurring when the operating point crosses the curve $\Gamma_\sLP$. For a fixed dilution rate $D=0.2$, increasing the input substrate concentration $S_{in}$ from zero yields a saddle–node bifurcation at the critical value $S_{in}=\sigma \approx 0.1687$. At this threshold, a pair of positive coexistence equilibria emerges, consisting of a LES equilibrium $\mathcal{E}_1^*$ and an unstable equilibrium $\mathcal{E}_2^*$.
\end{proof}
\begin{remark}
The theoretical construction of the operating diagram relies on the computation of the curve $\Gamma_\sLP$, which can in principle be obtained using symbolic computation tools such as \textsc{Maple}. This amounts to solving the system of three equations \cref{Eqs3SNProof}, involving the state variables $x_1$, $x_2$, and the operating parameters $S_{in}$ and $D$. Since this system involves three equations with four unknowns $(x_1, x_2, S_{in}, D)$, one can employ a parameter-driven iterative approach, using the solution from the previous step as an initial guess for \texttt{fsolve}, which effectively reduces the number of unknowns and facilitates computation \cite[Section 4]{FekihJMAA2025}.
In the present study, however, the operating diagrams are constructed using the numerical continuation package \textsc{MatCont} \cite{MATCONT2023}. This approach allows us to systematically trace the relevant bifurcation curves and provides a detailed characterization of the operating diagram.
\end{remark}

The operating diagram displayed in \cref{FigDOSSMort}(a) is partitioned into two regions separated by the curve $\Gamma_\sLP$. 
In region $\mathcal{J}_0$, no positive equilibrium exists; the washout equilibrium $\mathcal{E}_0$ is the unique equilibrium and is LES. 
In contrast, region $\mathcal{J}_1$ corresponds to a bistable regime where two additional equilibria appear: a LES coexistence equilibrium $\mathcal{E}_1^*$ and an unstable equilibrium $\mathcal{E}_2^*$. 
\subsection{Case with mortality}                 \label{SubSec-AvMort}
We now investigate the operating diagram of system \cref{ModelMutualism} when species mortality is included, that is, for $a_i \neq 0$, $i=1,2$. 
The analysis is carried out for the specific growth functions defined in \cref{SpeciFunc} and for the biological parameter values listed in \cref{Tab-Allpar}, with mortality rates fixed at $a_1=0.8$ and $a_2=1.5$.
The associated two-parameter bifurcation diagram in the $(S_{in},D)$-plane is computed numerically using the continuation package \textsc{MatCont} \cite{MATCONT2023}, which allows us to trace the bifurcation curves of equilibria and periodic solutions and to determine the qualitative dynamics in each region of the diagram. 
The resulting classification is summarized below.
\begin{proposition}
For the specific growth functions $f_1$ and $f_2$ defined in \cref{SpeciFunc}, and for the biological parameter values listed in \cref{Tab-Allpar} with mortality rates $a_1=0.8$ and $a_2=1.5$, the existence and local stability properties of the equilibria $\mathcal{E}_0$, $\mathcal{E}_1^*$, $\mathcal{E}_2^*$, and of the limit cycles $C_1$, $C_2$, $C_3$ of system \cref{ModelMutualism} are completely characterized within the regions $\mathcal{J}_0$, $\mathcal{J}_1$, and $\mathcal{J}_2$ of the operating diagram shown in \cref{FigDO}. 
The bifurcation curves separating these regions are listed in \cref{TabCurvBif}, 
the regional classification is detailed in \cref{Tab-RegionOD}, 
while the critical parameter values and associated normal form coefficients 
of the codimension-two bifurcations are summarized in \cref{Tab-VPDOMatc1}.
\end{proposition}
\begin{table}[ht]
\caption{Bifurcation curves of system \cref{ModelMutualism} illustrated in \cref{FigDO}, with their corresponding colors and types.} 
\label{TabCurvBif} 
\begin{center}
\begin{tabular}{ @{\hspace{1mm}}l@{\hspace{3mm}} l @{\hspace{3mm}} l @{\hspace{1mm}} }
Curve          &  Color    &  Type of bifurcation \\ \hline
$\Gamma_{LP}$  &  Blue     &  LP: Limit point (saddle-node) \\  
$\Gamma_H$     &  Red      &  H: Hopf bifurcation \\  
$\Gamma_{LPC}$ &  Green    &  LPC: Limit point of cycles \\
$\Gamma_{PD}$ &   Cyan     &  PD: Period-Doubling \\
\end{tabular}
\end{center}
\end{table}
\begin{table}[ht]
\caption{Existence and stability of equilibria and limit cycles in the different regions of \cref{FigDO}.}
\label{Tab-RegionOD}
\begin{center}
\begin{tabular}{lcccl}
\hline
Region & $\mathcal{E}_0$ & $\mathcal{E}^*_1$ & $\mathcal{E}^*_2$ & $(C_1, C_2, C_3)$ \\
\hline
$\mathcal{J}_0$                     & S &   &   &                    \\
\hline
$\mathcal{J}_1^0$                   & S & U & U &                    \\
$\mathcal{J}_1^{C_1^u}$             & S & U & U & $(U,\cdot,\cdot)$  \\
$\mathcal{J}_1^{C_1^s}$             & S & U & U & $(S,\cdot,\cdot)$  \\
$\mathcal{J}_1^{C_{12}^{su}}$      & S & U & U & $(S,U,\cdot)$      \\
$\mathcal{J}_1^{C_{123}^{sus}}$     & S & U & U & $(S,U,S)$          \\
\hline
$\mathcal{J}_2^0$                   & S & S & U &                    \\
$\mathcal{J}_2^{C_2^u}$             & S & S & U & $(\cdot, U, \cdot)$  \\
$\mathcal{J}_2^{C_1^s}$             & S & S & U & $(S,\cdot,\cdot)$  \\
$\mathcal{J}_2^{C_{23}^{us}}$            & S & S & U & $(\cdot, U,S)$\\ 
$\mathcal{J}_2^{C_{23}^{uu}}$            & S & S & U & $(\cdot, U,U)$\\  
\hline
\end{tabular}
\end{center}
\end{table}
\begin{figure}[!ht]
\setlength{\unitlength}{1.0cm}
\begin{center}
\begin{picture}(11.5,4.8)(0,0)
\put(-4,-6){\rotatebox{0}{\includegraphics[width=11cm,height=16cm]{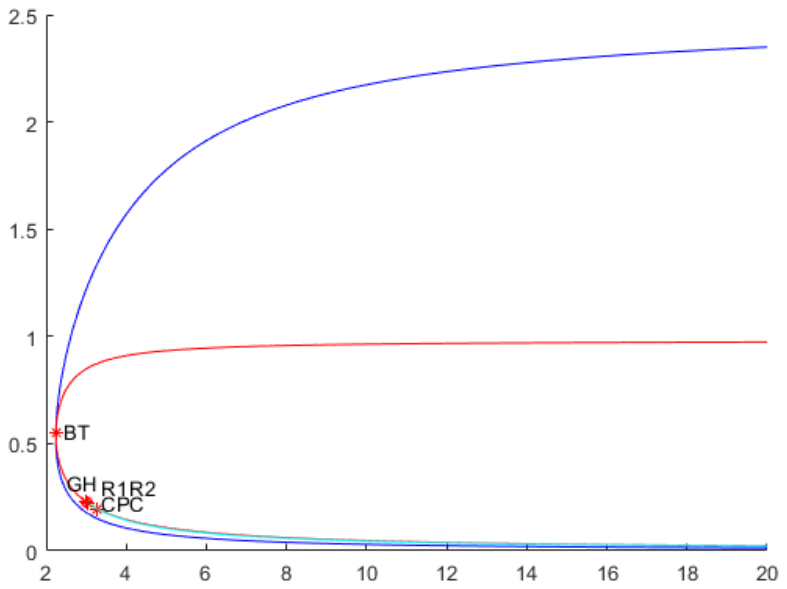}}}
\put(4,-6){\rotatebox{0}{\includegraphics[width=11cm,height=16cm]{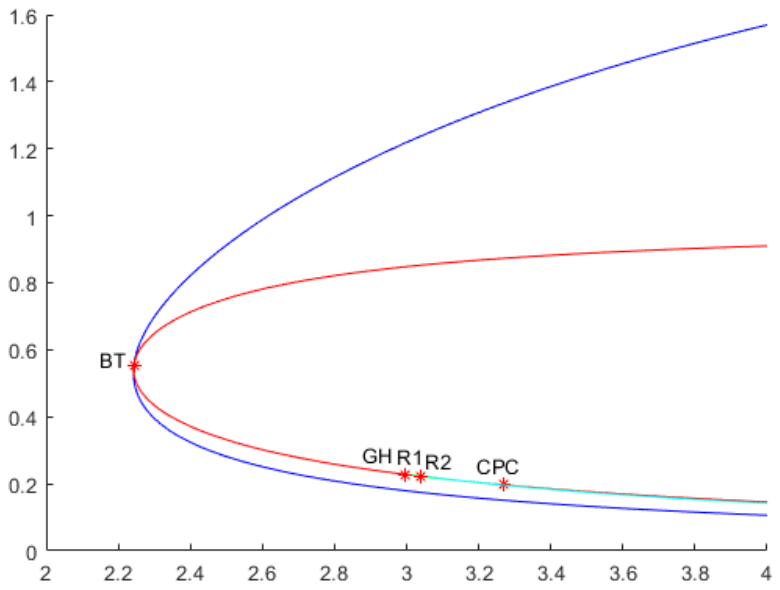}}}
\put(1.5,4.6){{\sc  $(a)$}}
\put(-1.4,4.5){{\sc  $D$}}
\put(4.9,4.2){{\sc  ${\color{blue}\Gamma_{LP}}$}}
\put(4.9,1.5){{\sc  ${\color{red}\Gamma_H}$}}
\put(-0.5,4){{\sc $\mathcal{J}_0$}}
\put(2,2.7){{\sc $\mathcal{J}_1$}}
\put(2,0.5){{\sc $\mathcal{J}_2$}}
\put(4.9,-0.35){{\sc  $S_{in}$}}
\put(9.5,4.55){{\sc $(b)$}}
\put(6.5,4.6){{\sc $D$}}
\put(8,4){{\sc $\mathcal{J}_0$}}
\put(10.6,3){{\sc $\mathcal{J}_1$}}
\put(10.6,1.5){{\sc $\mathcal{J}_2$}}
\put(12.9,4.4){{\sc  ${\color{blue}\Gamma_{LP}}$}}
\put(12.9,2.4){{\sc {\color{red}$\Gamma_H$}}}
\put(12.9,-0.35){{\sc $S_{in}$}}
\end{picture}\\
\begin{picture}(6.4,10.5)(0,0)
\put(-6,0){\rotatebox{0}{\includegraphics[width=8cm,height=14cm]{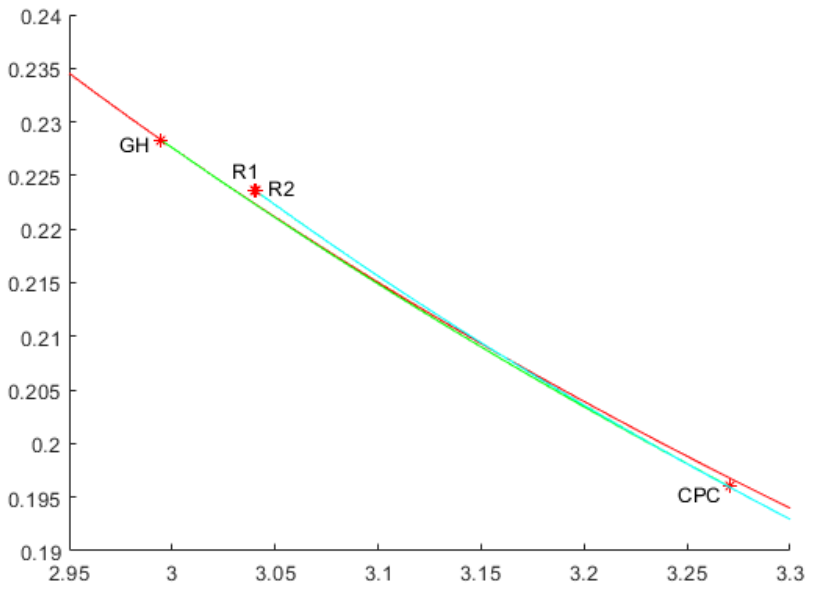}}}
\put(0,0){\rotatebox{0}{\includegraphics[width=7cm,height=14cm]{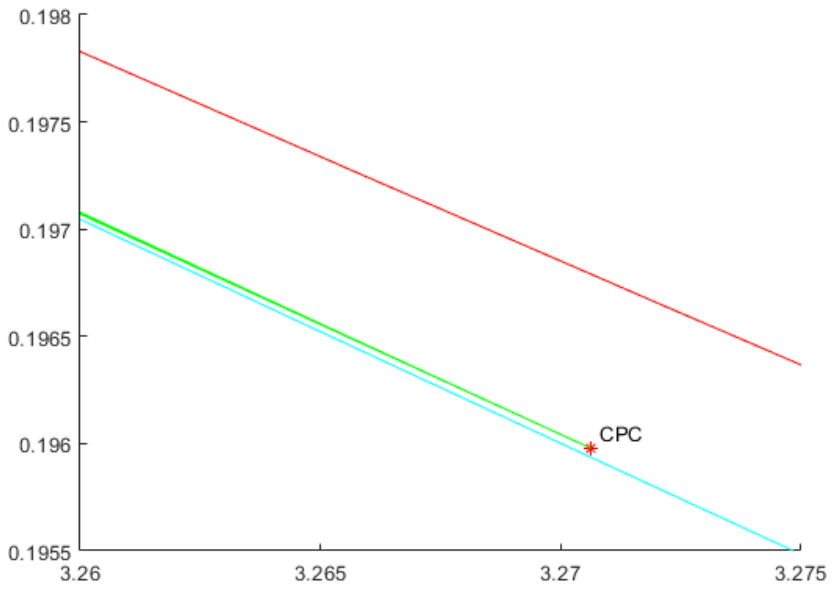}}}
\put(5,0){\rotatebox{0}{\includegraphics[width=7cm,height=14cm]{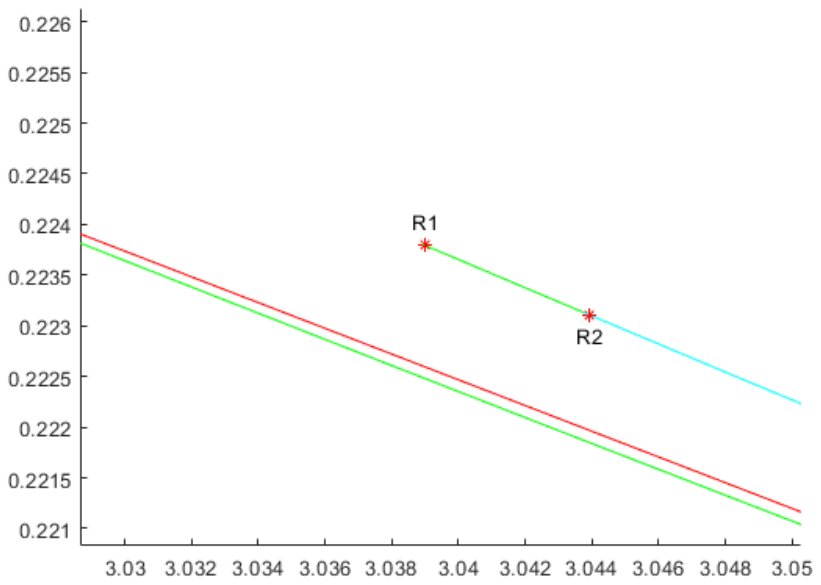}}}
\put(-2,9.2){{\sc  $(c)$}}
\put(-4.2,9.2){{\sc  $D$}}
\put(-0.3,5.8){{\sc  ${\color{red}\Gamma_H}$}}
\put(-4,7.3){{\sc  ${\color{green}\Gamma_{LPC}}$}}
\put(-3.75,7.55){\sc  {\color{green}\vector(1,1){0.4}}}
\put(-1,5.2){{\sc  ${\color{cyan}\Gamma_{PD}}$}}
\put(-0.3,5.2){\sc  {\color{cyan}\vector(1,0){0.6}}}
\put(0.5,4.95){{\sc  $S_{in}$}}
\put(3.7,9.2){{\sc  $(d)$}}
\put(1.6,9.2){{\sc  $D$}}
\put(5,6.9){{\sc  ${\color{red}\Gamma_H}$}}
\put(2,7.5){{\sc  ${\color{green}\Gamma_{LPC}}$}}
\put(5,5.4){{\sc  ${\color{cyan}\Gamma_{PD}}$}}
\put(5.65,4.95){{\sc  $S_{in}$}}
\put(8.4,9.2){{\sc  $(e)$}}
\put(6.6,9.2){{\sc  $D$}}
\put(7,7.3){{\sc  ${\color{red}\Gamma_H}$}}
\put(7,6.6){{\sc  ${\color{green}\Gamma_{LPC}}$}}
\put(8.8,7.2){{\sc  ${\color{green}\Gamma_{LPC}}$}}
\put(10,6.45){{\sc  ${\color{cyan}\Gamma_{PD}}$}}
\put(10.5,4.95){{\sc  $S_{in}$}}
\end{picture}
\end{center}
 \vspace{-4.8cm}
 \caption{\textsc{MatCont}: 
(a) Operating diagram of system \cref{ModelMutualism}. 
(b) Zoom of the region containing the BT, GH, \texttt{R1}, \texttt{R2}, and CPC bifurcation points. 
(c) Magnified view of the GH–\texttt{R1}–\texttt{R2}–CPC interaction region. 
(d) Detailed zoom around the CPC bifurcation point. 
(e) Detailed zoom around the \texttt{R1}–\texttt{R2} resonance points.}\label{FigDO}
\end{figure}
\begin{table}[ht]
\caption{Operating parameter values and associated normal form coefficients at the codimension-two bifurcation points shown in \cref{FigDO}. 
BT denotes a Bogdanov--Takens bifurcation, GH a generalized Hopf bifurcation, \texttt{R1} a $1:1$ resonance bifurcation, and CPC a cusp bifurcation of cycles.}
\label{Tab-VPDOMatc1}
\begin{center}
\begin{tabular}{lll}
\hline
Bif. & $(S_{in},D)$ & Normal form coefficients \\ 
\hline
BT          & $(2.243,\,0.550)$ & $(a,b)\approx(-0.786,\,-1.076)$ \\ 
GH          & $(2.995,\,0.228)$ & $l_2\approx-117.07$ \\    
\texttt{R1} & $(3.0396,\,0.2237)$ & $a\,b\approx6.810\times 10^{6}$ \\  
\texttt{R2} & $(3.0437,\,0.2232)$ & $(a,b)\approx(55.43,\,-1098)$ \\ 
CPC         & $(3.271,\,0.196)$ & $c \approx 27.419$ \\
\hline
\end{tabular}
\end{center}
\end{table}
\begin{proof}
Recall from \cref{Sec-AnalMod} that $\mathcal{E}_0$ always exists and is LES. 
Crossing the curve $\Gamma_{LP}$ gives rise to two positive equilibria through a saddle–node bifurcation. 
Crossing the Hopf curve $\Gamma_H$ generates a family of periodic solutions.
At the intersection of $\Gamma_{LP}$ and $\Gamma_H$, a Bogdanov–Takens (BT) bifurcation point is detected. 
Moreover, the Hopf curve contains a generalized Hopf (GH) point, where the first Lyapunov coefficient vanishes and the criticality of the Hopf bifurcation changes. 
From this GH point, a branch of limit point of cycles bifurcations $\Gamma_{LPC}$ emanates, corresponding to the creation or annihilation of a pair of periodic orbits.

In addition to these curves, the continuation of periodic orbits reveals the presence of a period-doubling curve $\Gamma_{PD}$ in the $(S_{in},D)$-plane. 
This curve corresponds to flip bifurcations of limit cycles, occurring when a Floquet multiplier crosses the unit circle at $-1$. 
Numerically, $\Gamma_{PD}$ is initially located very close to the limit point of cycles curve $\Gamma_{LPC}$ in a restricted parameter region. 
However, after the termination of $\Gamma_{LPC}$ at the CPC point, the period-doubling curve persists and bends toward the Hopf curve $\Gamma_H$, remaining entirely in the region where periodic solutions exist (see \cref{FigDO}).
Its turning point corresponds to the codimension-two resonance point \texttt{R2}, located at the maximal extension of $\Gamma_{PD}$ in the $(S_{in},D)$-plane (see \cref{FigDO}).

\textsc{MatCont} also detects other codimension-two bifurcations organizing the dynamics of periodic solutions, notably cusp bifurcations of cycles (CPC), and $1:1$ and $1:2$ resonance bifurcations (\texttt{R1} and \texttt{R2}), also shown in \cref{FigDO}. 
A CPC bifurcation corresponds to the tangential merging of two branches of limit point of cycles \cite{Kuznetsov2004}. 
A $1:1$ resonance bifurcation (\texttt{R1}) occurs when a Floquet multiplier of a periodic orbit crosses the unit circle at $+1$ with multiplicity greater than one. 
A $1:2$ resonance bifurcation (\texttt{R2}) corresponds to the presence of a double multiplier at $-1$ on the unit circle.  

The coefficients listed in \cref{Tab-VPDOMatc1} correspond to the reduced normal forms computed by \textsc{MatCont} \cite[Section~3]{KuznetsovSiam1999}. 
Their nonvanishing verifies that the detected codimension-two bifurcations satisfy the required non-degeneracy conditions and therefore admit a generic unfolding. 
Their sign determines the local orientation of the bifurcation branches in parameter space, whereas their magnitude has no qualitative significance and merely reflects the numerical scaling of the normal form (see \cite{Govaerts2000,Kuznetsov2004}).

Finally, \cref{Tab-RegionOD} provides a detailed refinement of the operating diagram, indicating the existence and stability of equilibria and limit cycles in each region of the $(S_{in},D)$-plane. 
Within the regions $\mathcal{J}_1$ and $\mathcal{J}_2$, several subregions may occur where up to three limit cycles coexist, with at most two of them being stable. 
These cycles can appear or disappear through Hopf (H), limit point of cycles (LPC), or period-doubling (PD) bifurcations. 
In addition, homoclinic (Hom) bifurcations, which are global in nature, are detected and analyzed in the subsequent subsection devoted to the detailed bifurcation structure (see \cref{Sec-BD}).
\end{proof}

Altogether, the operating diagram provides a comprehensive overview of the qualitative dynamics of system \cref{ModelMutualism} in the $(S_{in},D)$-plane, illustrating how washout, coexistence, and oscillatory regimes are organized through both local and codimension-two bifurcations. 
The detailed one-parameter bifurcation analysis in the following section will further reveal the internal organization of these regions and explicitly show how periodic solutions are structured and connected.
\section{One-parameter bifurcation diagrams and numerical simulations} \label{Sec-BD}
In this section, we complement the operating diagram analysis by investigating one-parameter bifurcation diagrams computed with \textsc{MatCont}, using the input substrate concentration $S_{in}$ as the continuation parameter for several fixed values of the dilution rate $D$. Analogous diagrams can also be obtained by taking $D$ as the continuation parameter.
Each selected value of $D$ corresponds to a horizontal slice of the operating diagram in the $(S_{in},D)$-plane, allowing a detailed examination of the system dynamics within the corresponding region.

We first consider the case $D=0.195$, for which the horizontal line lies below the cusp bifurcation of cycles (CPC) located at $(S_{in},D)\approx(3.271,\,0.196)$.  
Next, we analyze the case $D=0.2$, corresponding to a horizontal section situated between the CPC point and the $1\!:\!1$ resonance point \texttt{R1}, whose coordinates are $(S_{in},D)\approx(3.0396,\,0.2237)$.  
We also examine the case $D=0.22$, which lies between the CPC and \texttt{R1} points but differs in the ordering of bifurcations along $S_{in}$: the sequence is LPC–H–PD–LPC, instead of LPC–PD–LPC–H.  
For $D=0.224$, the horizontal line lies between the \texttt{R1} point and the generalized Hopf (GH) bifurcation at $(S_{in},D)\approx(2.995,\,0.228)$.  
Finally, the case $D=0.23$ corresponds to a line above the GH point.

These one-parameter bifurcation diagrams provide a refined description of the asymptotic behavior of the system within each region of the operating diagram. 
In particular, they reveal the precise sequence of qualitative transitions, including the emergence or disappearance of equilibria and limit cycles, as $S_{in}$ varies. This analysis clarifies the dynamical mechanisms underlying the structure of the operating diagram.
\subsection{A unique limit cycle for \texorpdfstring{$D=0.195$}{D=0.195} without an LPC bifurcation}
The following result shows that, for $D=0.195$, the system exhibits a unique branch of periodic solutions which does not undergo any Limit Point of Cycles (LPC) bifurcation.
\begin{proposition}
Fix $D=0.195$. For the specific growth functions $f_1$ and $f_2$ and the biological parameter values used in \cref{FigDO}, the one–parameter bifurcation diagram of \cref{ModelMutualism} with respect to $S_{in}$ is displayed in \cref{Fig-DBD0195}. The existence and local stability of all steady states and periodic solutions according to $S_{in}$ are summarized in \cref{Tab-DB0195}.
\end{proposition}
\begin{table}[ht]
\caption{Critical values $\sigma_i$, $i=1,\ldots,4$, of the input substrate concentration $S_{in}$ for $D=0.195$. For each value, the table lists the bifurcation type, the corresponding equilibrium coordinates $(S,x_1,x_2)$ or the period $T$ of periodic solutions, and the existence and stability of equilibria and periodic orbits (S: stable, U: unstable, blank: non-existent).}\label{Tab-DB0195}
\begin{center}
\begin{tabular}{ @{\hspace{1mm}}l@{\hspace{2mm}} @{\hspace{2mm}}l@{\hspace{2mm}}  
                 @{\hspace{2mm}}l@{\hspace{2mm}} @{\hspace{1mm}}l@{\hspace{1mm}}  
                 @{\hspace{1mm}}l@{\hspace{1mm}} @{\hspace{1mm}}l@{\hspace{1mm}} 
                 @{\hspace{1mm}}l@{\hspace{1mm}} @{\hspace{1mm}}l@{\hspace{1mm}} }
\hline
$S_{in}$ & Bif. &  $(S,x_1,x_2)$ / $T$ 
& $\mathcal{E}_0$ & $\mathcal{E}_1^*$ & $\mathcal{E}_2^*$ & $C_1$ & Notes \\ 
\hline
$(0,\sigma_1)$         
& -- 
& -- 
& S &  &  &  & \sl{$\mathcal{J}_0$}\\

$\sigma_1 \approx 2.883$ 
& LP 
& \sc{$(0.669,0.190,0.143)$} 
&  &  &  &  & \sl{$\mathcal{E}_1^*=\mathcal{E}_2^*$} \\

$(\sigma_1,\sigma_2)$         
& -- 
& -- 
& S & U & U &  & \sl{$\mathcal{J}_1^0$} \\

$\sigma_2 \approx 3.277$ 
& Hom 
& $T \to +\infty$ 
&  &  &  &  & \sl{Homoclinic} \\

$(\sigma_2,\sigma_3)$         
& -- 
& -- 
& S & U & U & U & \sl{$\mathcal{J}_1^{C_1^u}$} \\

$\sigma_3 \approx 3.280$ 
& PD 
& $T=17.77 $ 
&  &  &  &  & \sl{Period-Doubling} \\
 
$(\sigma_3,\sigma_4)$         
& -- 
& -- 
& S & U & U & S & \sl{$\mathcal{J}_1^{C_1^s}$} \\

$\sigma_4 \approx 3.289$ 
& H 
& \sc{$(0.338,0.243,0.196)$} 
&  &  &  &  & \sl{Supercritical Hopf ($l_1 \approx -0.347$)}\\
$(\sigma_4,+\infty)$         
& -- 
& -- 
& S & S & U &   & \sl{$\mathcal{J}_2^0$} \\
\hline
\end{tabular}
\end{center}
\end{table}
\begin{figure}[!ht]
\setlength{\unitlength}{1.0cm}
\begin{center}
\begin{picture}(6.5,7)(0,0)
\put(-5,0){\rotatebox{0}{\includegraphics[width=6.5cm,height=10cm]{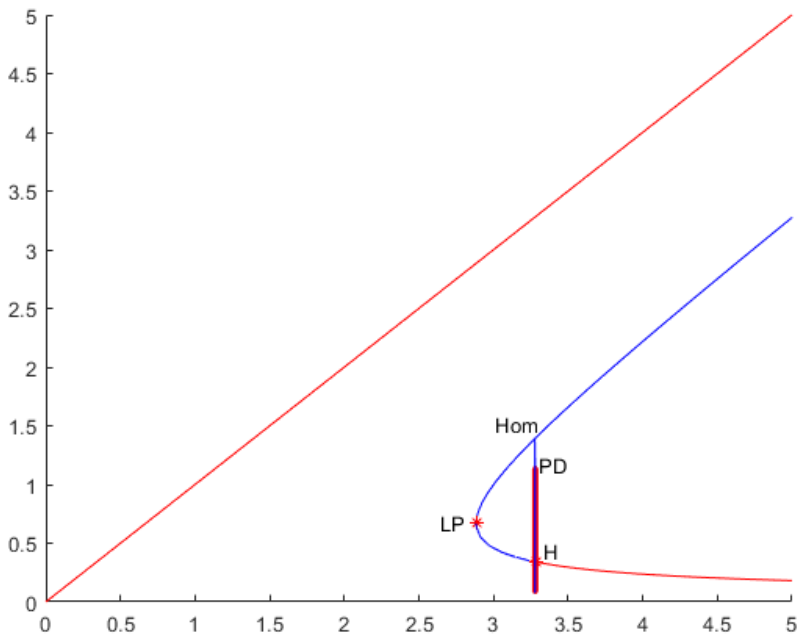}}}
\put(0,0){\rotatebox{0}{\includegraphics[width=6.5cm,height=10cm]{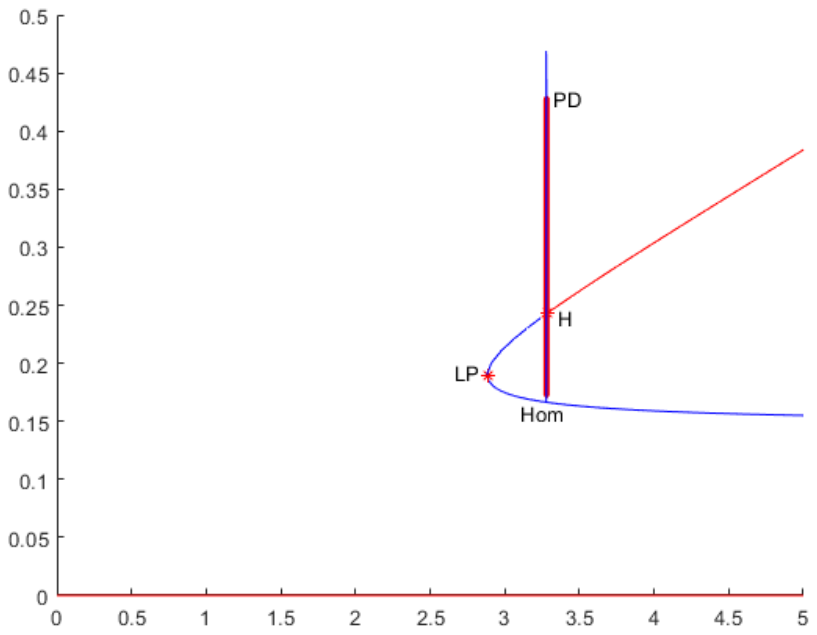}}}
\put(5,0){\rotatebox{0}{\includegraphics[width=6.5cm,height=10cm]{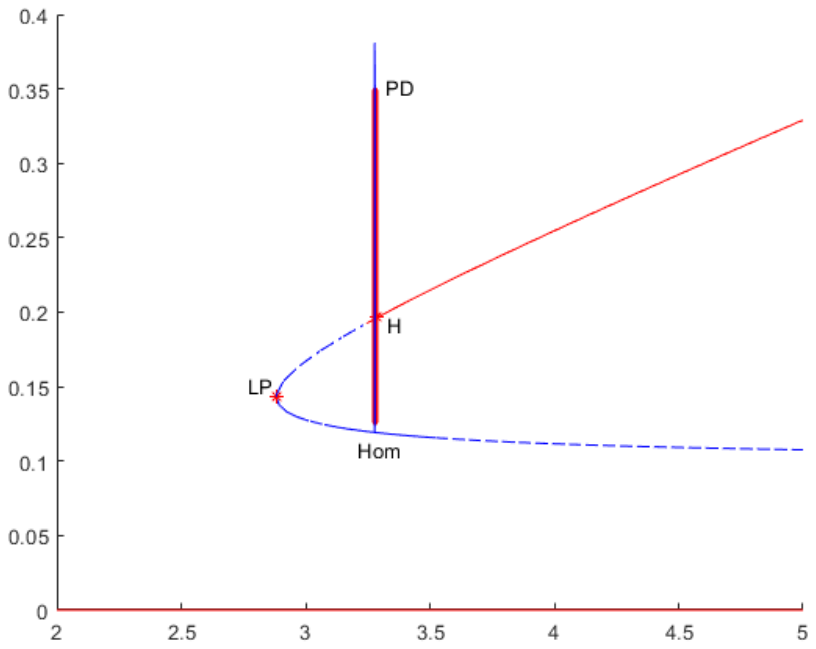}}}
\put(-1.8,6.6){{\sc $(a)$}}
\put(-3.5,6.6){{\sc $S$}}
\put(0.35,3.4){{\sc  $S_{in}$}}
\put(-1.65,5.4){{\sc {\color{red} $\mathcal{E}_0$}}}
\put(-0.3,4.85){{\sc {\color{blue} $\mathcal{E}_2^*$}}}
\put(-1.5,3.45){{\sc {\color{blue} $\mathcal{E}_1^*$}}}
\put(-0.4,3.7){{\sc {\color{red} $\mathcal{E}_1^*$}}}
\put(3.2,6.6){{\sc $(b)$}}
\put(1.5,6.6){{\sc $x_1$}}
\put(5.35,3.4){{\sc  $S_{in}$}}
\put(3.3,3.5){{\sc {\color{red} $\mathcal{E}_0$}}}
\put(4.7,4.6){{\sc {\color{blue} $\mathcal{E}_2^*$}}}
\put(4.7,5.85){{\sc {\color{red} $\mathcal{E}_1^*$}}}
\put(3.5,5){{\sc {\color{blue} $\mathcal{E}_1^*$}}}
\put(8.4,6.6){{\sc $(c)$}}
\put(6.5,6.6){{\sc $x_2$}}
\put(10.35,3.4){{\sc  $S_{in}$}}
\put(8.3,3.5){{\sc {\color{red} $\mathcal{E}_0$}}}
\put(7.5,5){{\sc {\color{blue} $\mathcal{E}_1^*$}}}
\put(9,5.8){{\sc {\color{red} $\mathcal{E}_1^*$}}}
\put(9,4.45){{\sc {\color{blue} $\mathcal{E}_2^*$}}}
\end{picture}\\
\begin{picture}(6.5,4)(0,0)
\put(-5,0){\rotatebox{0}{\includegraphics[width=6.5cm,height=10cm]{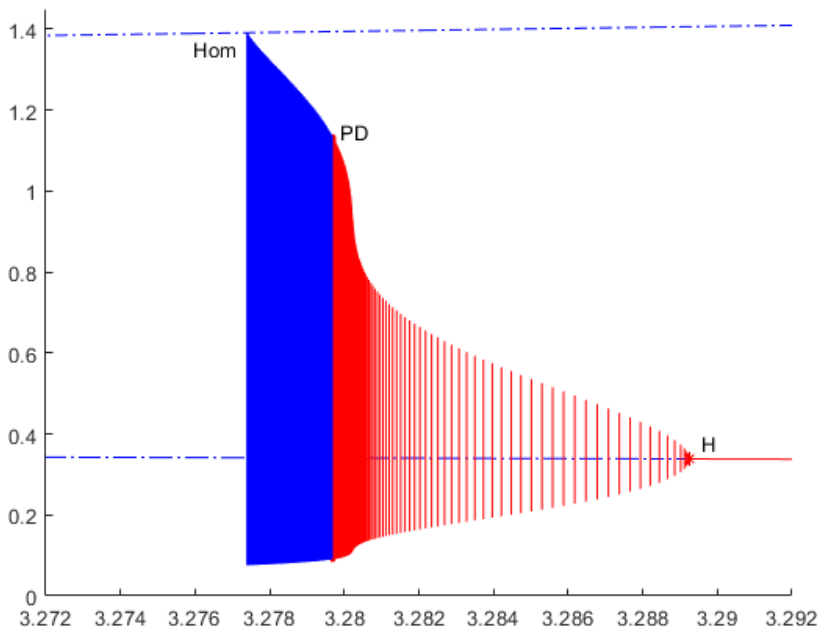}}}
\put(0,0){\rotatebox{0}{\includegraphics[width=6.5cm,height=10cm]{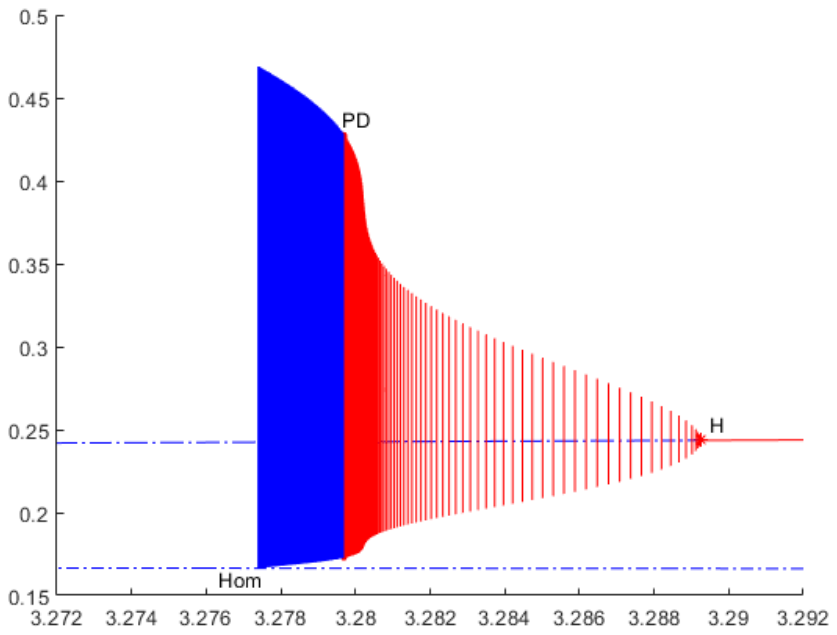}}}
\put(5,0){\rotatebox{0}{\includegraphics[width=6.5cm,height=10cm]{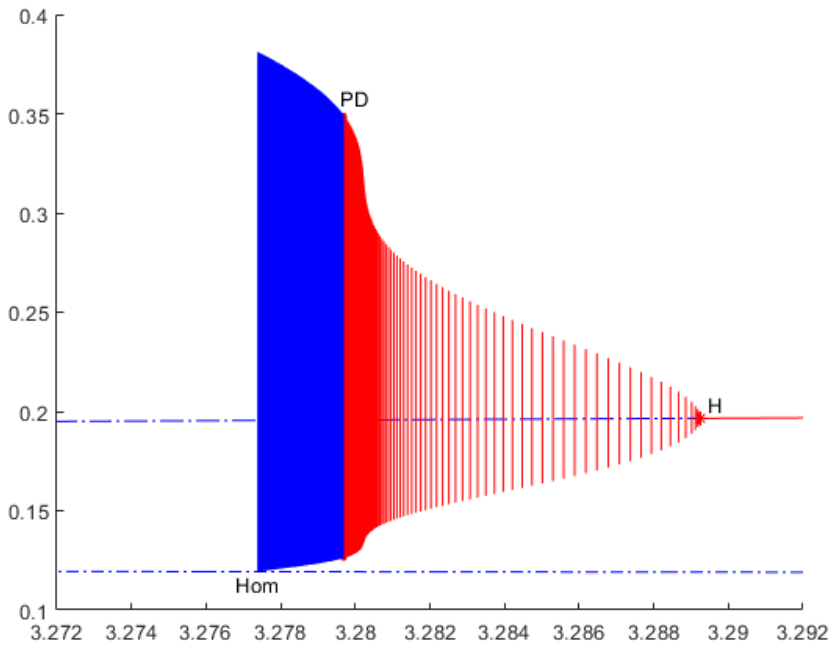}}}
\put(-1.8,6.8){{\sc $(d)$}}
\put(-3.5,6.7){{\sc $S$}}
\put(0.35,3.4){{\sc  $S_{in}$}}
\put(-0.3,6.4){{\sc {\color{blue} $\mathcal{E}_2^*$}}}
\put(-3.2,4.35){{\sc {\color{blue} $\mathcal{E}_1^*$}}}
\put(0.1,4.35){{\sc {\color{red} $\mathcal{E}_1^*$}}}
\put(3.2,6.8){{\sc $(e)$}}
\put(1.5,6.7){{\sc $x_1$}}
\put(5.35,3.4){{\sc  $S_{in}$}}
\put(1.8,4.45){{\sc {\color{blue} $\mathcal{E}_1^*$}}}
\put(5.1,4.45){{\sc {\color{red} $\mathcal{E}_1^*$}}}
\put(4.2,3.7){{\sc {\color{blue} $\mathcal{E}_2^*$}}}
\put(8.2,6.8){{\sc $(f)$}}
\put(6.5,6.7){{\sc $x_2$}}
\put(10.35,3.4){{\sc  $S_{in}$}}
\put(6.8,4.55){{\sc {\color{blue} $\mathcal{E}_1^*$}}}
\put(10.1,4.55){{\sc {\color{red} $\mathcal{E}_1^*$}}}
\put(9.2,3.7){{\sc {\color{blue} $\mathcal{E}_2^*$}}}
\end{picture}
\end{center}
\vspace{-3.3cm}
\caption{\textsc{MatCont} results for $D=0.195$. (a--c) One-parameter bifurcation diagrams of system \cref{ModelMutualism} with respect to $S_{in}$, displaying the equilibrium branches of the substrate $S$ and the biomasses $x_1$ and $x_2$. Stable solutions are shown in red and unstable ones in blue. (d--f) Enlarged views near critical parameter values, highlighting the Hopf (H), homoclinic (Hom), and period-doubling (PD) bifurcations, together with the associated branches of periodic solutions.}
\label{Fig-DBD0195}
\end{figure}
\begin{proof}
From \cref{PropStabE0}, the washout equilibrium $\mathcal{E}_0$ is always LES.  
As $S_{in}$ increases from zero, the system undergoes a Limit Point (LP) bifurcation at $S_{in} = \sigma_1$, giving rise to two positive equilibria, $\mathcal{E}_1^*$ and $\mathcal{E}_2^*$. 
The equilibrium $\mathcal{E}_2^*$ remains unstable, while $\mathcal{E}_1^*$ becomes LES only for $S_{in} > \sigma_4$, after the Hopf bifurcation.

As $S_{in}$ decreases past $\sigma_4$, $\mathcal{E}_1^*$ loses stability through a supercritical Hopf bifurcation (first Lyapunov coefficient $l_1 < 0$), giving rise to a stable limit cycle $C_1$.  
Continuing the branch of $C_1$ with respect to $S_{in}$ toward smaller values, it undergoes a period-doubling (PD) bifurcation at $S_{in} = \sigma_3$, producing an unstable periodic orbit.  
Further decreasing $S_{in}$, this unstable cycle grows in amplitude and ultimately disappears in a homoclinic bifurcation at $S_{in} = \sigma_2$, where its period diverges to $+\infty$ (see \Cref{Fig-TnearHom}(a)).
\end{proof}
\begin{figure}[!ht]        
\setlength{\unitlength}{1.0cm}
\begin{center}
\begin{picture}(10,7.2)(0,0)
\put(-3.5,0){\rotatebox{0}{\includegraphics[width=10cm,height=10cm]{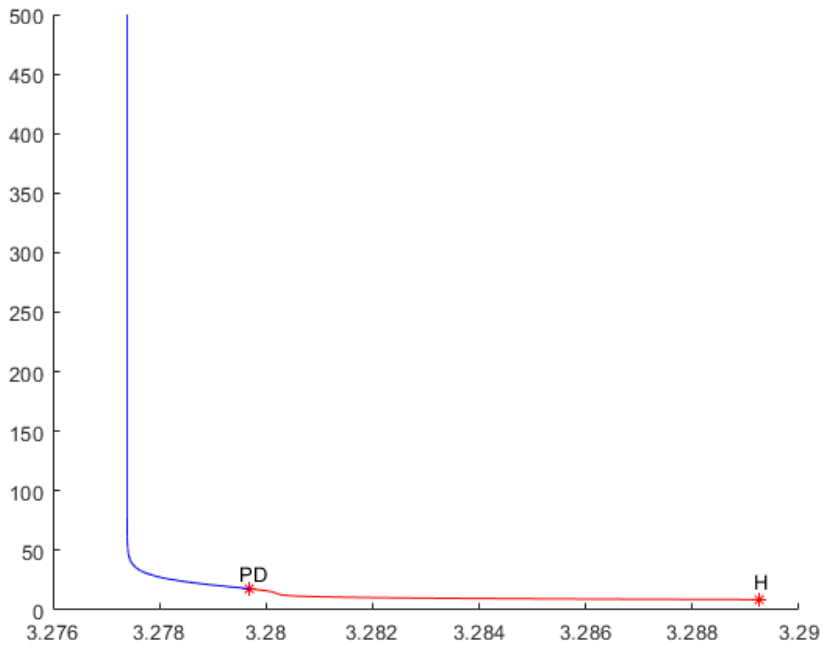}}}
\put(3.5,-0.2){\rotatebox{0}{\includegraphics[width=10cm,height=10.3cm]{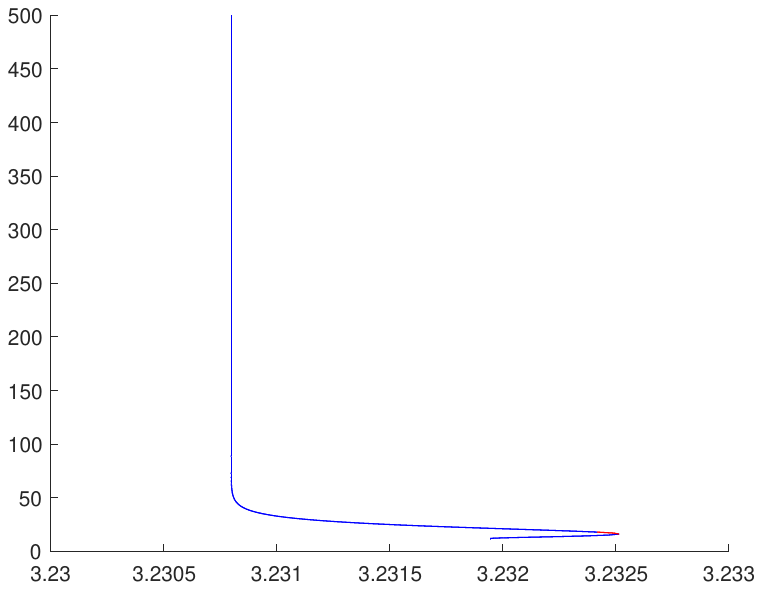}}}
\put(-1.3,6.7){{\sc  $T$}}
\put(1.4,6.7){{\sc  $(a)$}}
\put(4.7,3.35){{\sc  $S_{in}$}}
\put(-0.7,5.8){{\sc {\color{blue} $C_1$}}}
\put(1.6,3.55){{\sc {\color{red} $C_1$}}}
\put(-0.9,3){{\sc $\sigma_2$}}
\put(0.1,3){{\sc $\sigma_3$}}
\put(4.2,3){{\sc $\sigma_4$}}
\put(6,6.6){{\sc  $T$}}
\put(8.4,6.7){{\sc  $(b)$}}
\put(11.5,3.4){{\sc  $S_{in}$}}
\put(7.4,5.8){{\sc {\color{blue} $C_3$}}}
\put(10.3,3.65){{\sc {\color{red} $C_3$}}}
\put(9.8,3){{\sc {\color{blue} $C_2$}}}
\put(9.8,3.3){\sc {\color{blue} $\uparrow$}}
\put(7.2,3){{\sc $\sigma_2$}}
\put(9.2,3){{\sc $\sigma_3$}}
\put(10.5,3){{\sc $\sigma_5$}}
\end{picture}
\end{center}
\vspace{-3.2cm}
\caption{Time period $T$ of limit cycle solutions as a function of $S_{in}$ near the homoclinic bifurcation at $S_{in}=\sigma_2$. 
(a) $D=0.195$: the branch $C_1$ emerges stable at $S_{in}=\sigma_4$ (H) and loses stability at $S_{in}=\sigma_3$ (PD); its period diverges as $S_{in}\to\sigma_2^+$. 
(b) $D=0.2$: the branch $C_2$ emerges unstable at $S_{in}=\sigma_3$ (LPC) and collides with $C_3$ at $S_{in}=\sigma_5$ (LPC); decreasing $S_{in}$, $C_3$ loses stability at $S_{in}=\sigma_4$ (PD), and $T\to+\infty$ as $S_{in}\to\sigma_2^+$. Stable and unstable periodic orbits are shown in red and blue.}\label{Fig-TnearHom}
\end{figure}
\begin{remark}
In all bifurcation diagrams, stable equilibria and periodic orbits are shown in red, and unstable ones in blue. 
The labels \textsc{LP} (saddle--node), Hom (homoclinic), \textsc{PD} (period-doubling), and \textsc{H} (Hopf) denote the corresponding bifurcation types.
\end{remark}
The preceding analysis allows us to summarize the qualitative dynamics of the system:
For $S_{in} < \sigma_2$, the washout equilibrium $\mathcal{E}_0$ is the unique local attractor.  
For $S_{in} \in (\sigma_2, \sigma_3)$, only the unstable periodic orbit exists, while $\mathcal{E}_0$ remains the unique local attractor.  
For $S_{in} \in (\sigma_3, \sigma_4)$, the system exhibits bistability between $\mathcal{E}_0$ and the stable limit cycle $C_1$.  
For $S_{in} > \sigma_4$, the positive equilibrium $\mathcal{E}_1^*$ becomes LES, while $C_1$ disappears, leading to bistability between $\mathcal{E}_0$ and $\mathcal{E}_1^*$.
\subsection{Up to three limit cycles for \texorpdfstring{$D=0.2$}{D=0.2} with two LPC bifurcations}
In this subsection, we fix the dilution rate at $D=0.2$. 
The corresponding one-parameter bifurcation diagram with respect to the input substrate concentration $S_{in}$ is shown in \cref{Fig-DB-D02}, where the equilibrium substrate concentration $S$ is plotted along the vertical axis. 
This analysis reveals the coexistence of multiple limit cycles and the occurrence of two limit point of cycles (LPC) bifurcations. 
The main dynamical features are summarized in the following proposition.
\begin{proposition}
Fix $D=0.2$. For the specific growth functions $f_1$ and $f_2$ and the biological parameter values used in \cref{FigDO}, the one–parameter bifurcation diagram of \cref{ModelMutualism} with respect to $S_{in}$ is displayed in \cref{Fig-DB-D02}. The existence and local stability of all steady states and periodic solutions as $S_{in}$ varies are summarized in \cref{Tab-DB02}.
\end{proposition}
\begin{table}[ht]
\caption{Critical values $\sigma_i$, $i=1,\ldots,6$, of the input substrate concentration $S_{in}$ for $D=0.2$. For each bifurcation value, the table reports the bifurcation type, the corresponding equilibrium coordinates $(S,x_1,x_2)$ or the period $T$ of periodic solutions, and the existence and stability of equilibria and limit cycles. 
The columns $E_0$, $\mathcal{E}_1^*$, $\mathcal{E}_2^*$, and $(C_1, C_2, C_3)$ indicate their presence and stability (S: stable, U: unstable, blank: not present).}\label{Tab-DB02}
\begin{center}
\begin{tabular}{ @{\hspace{1mm}}l@{\hspace{2mm}} @{\hspace{2mm}}l@{\hspace{2mm}}  
                 @{\hspace{2mm}}l@{\hspace{2mm}} @{\hspace{2mm}}l@{\hspace{2mm}}   
                 @{\hspace{2mm}}l@{\hspace{2mm}} @{\hspace{2mm}}l@{\hspace{2mm}} 
                @{\hspace{2mm}}l@{\hspace{2mm}} @{\hspace{2mm}}l@{\hspace{1mm}} }  
\hline
$S_{in}$ & Bif. &  $(S,x_1,x_2)$ / $T$ 
& $\mathcal{E}_0$ & $\mathcal{E}_1^*$ & $\mathcal{E}_2^*$ & $(C_1, C_2, C_3)$ & Notes \\ 
\hline
$(0,\sigma_1)$           
& -- 
& -- 
& S &  &  &  & \sl{$\mathcal{J}_0$} \\

$\sigma_1 \approx 2.8504$ 
& LP 
& \sc{$(0.665,0.191,0.144)$} 
&  &  &  &  & \sl{$\mathcal{E}_1^*=\mathcal{E}_2^*$} \\

$(\sigma_1,\sigma_2)$    
& -- 
& -- 
& S & U & U &  & $\mathcal{J}_1^0$ \\

$\sigma_2 \approx 3.2307$ 
& Hom 
& $T\to +\infty$ 
&  &  &  &  & \sl{Homoclinic} \\

$(\sigma_2,\sigma_3)$    
& -- 
& -- 
& S & U & U & $( \cdot, \cdot , U)$ & \sl{$\mathcal{J}_1^{C_1^u}$} \\

$\sigma_3 \approx 3.2319$ 
& LPC 
& $T=11.83$ 
&  &  &  &   & {\sl Collision of $C_1$ and $C_2$}  \\ 

$(\sigma_3,\sigma_4)$    
& -- 
& -- 
& S & U & U & $(S, U, U)$ & \sl{$\mathcal{J}_1^{C_{123}^{suu}}$} \\

$\sigma_4 \approx 3.2323$ 
& PD 
& $T=18.47$ 
&  &  &  &   &  \sl{Period-Doubling} \\

$(\sigma_4,\sigma_5)$    
& -- 
& -- 
& S & U & U & $(S, U, S)$ & \sl{$\mathcal{J}_1^{C_{123}^{sus}}$} \\

$\sigma_5 \approx 3.2325$ 
& LPC 
& $T=16.11$ 
&  &  &  &  &  {\sl Collision of $C_2$ and $C_3$} \\

$(\sigma_5,\sigma_6)$    
& -- 
& -- 
& S & U & U & $(S, \cdot, \cdot)$ & {\sl $\mathcal{J}_1^{C_1^s}$} \\

$\sigma_6 \approx 3.2381$ 
& H 
& \sc{$(0.341,0.244,0.197)$} 
&  &  &  &  & {\sl Hopf ($l_1 \approx -0.294$)} \\

$(\sigma_6,+\infty)$     
& -- 
& -- 
& S & S & U &  & {\sl $\mathcal{J}_2^0$} \\
\hline
\end{tabular}
\end{center}
\end{table}
\begin{figure}[!ht]
\setlength{\unitlength}{1.0cm}
\begin{center}
\begin{picture}(8,7)(0,0)
\put(-3.5,0){\rotatebox{0}{\includegraphics[width=8cm,height=9cm]{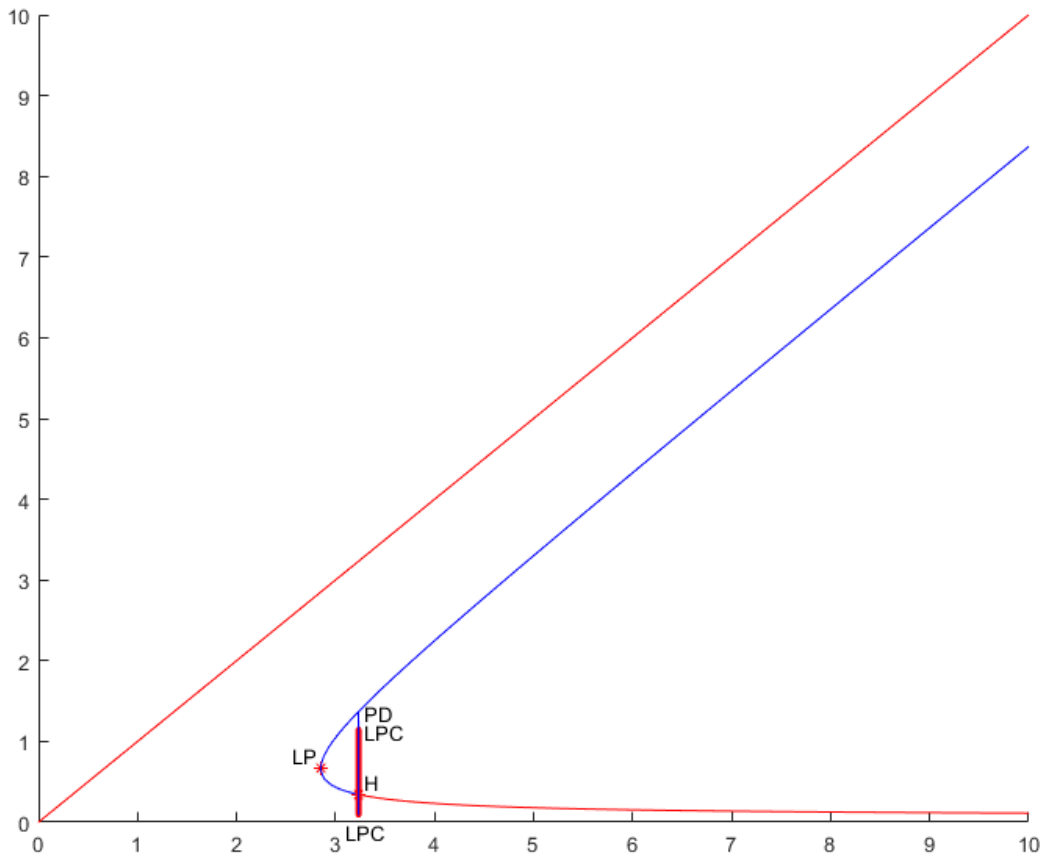}}}
\put(3.5,0){\rotatebox{0}{\includegraphics[width=8cm,height=9cm]{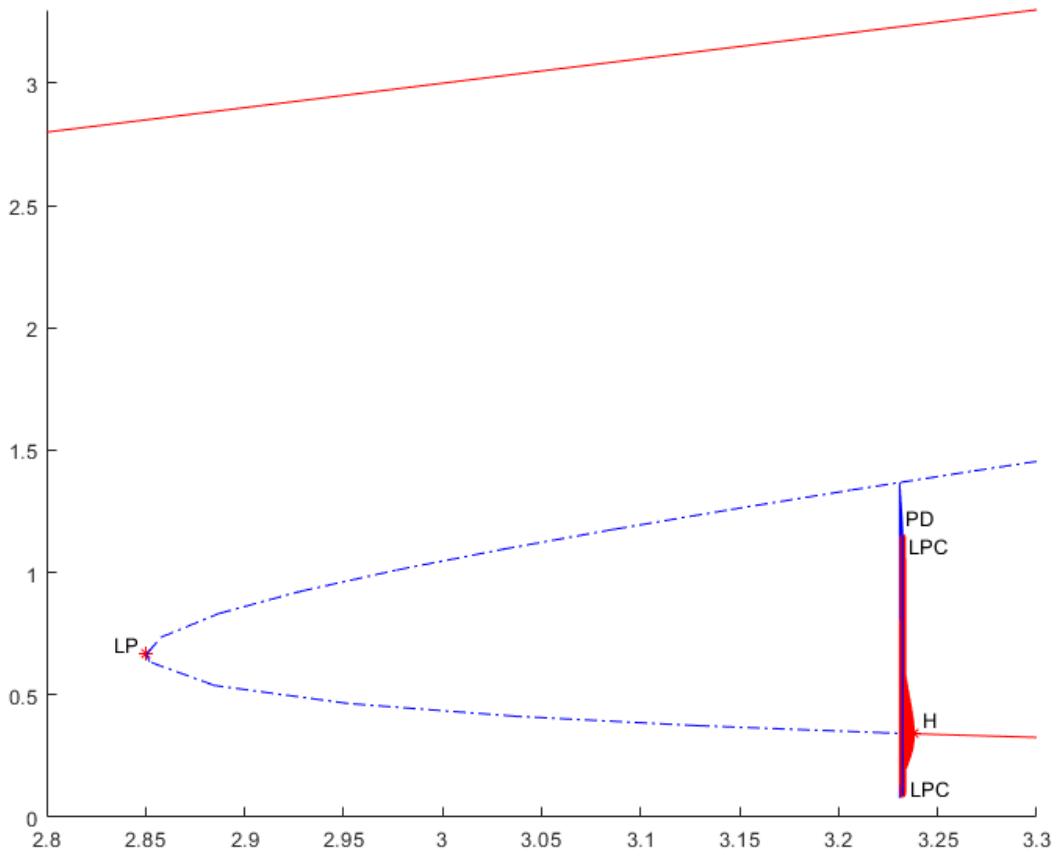}}}
\put(0.2,6.6){{\sc $(a)$}}
\put(-2.4,6.6){{\sc $S$}}
\put(3.55,2.5){{\sc  $S_{in}$}}
\put(-0.95,2.45){{\sc {\color{blue} $\mathcal{E}_1^*$}}}
\put(1.5,5.4){{\sc {\color{red} $\mathcal{E}_0$}}}
\put(2.5,2.7){{\sc {\color{red} $\mathcal{E}_1^*$}}}
\put(1.5,4.75){{\sc {\color{blue} $\mathcal{E}_2^*$}}}
\put(7.2,6.6){{\sc $(b)$}}
\put(4.5,6.6){{\sc $S$}}
\put(10.9,2.5){{\sc  $S_{in}$}}
\put(6,2.85){{\sc {\color{blue} $\mathcal{E}_1^*$}}}
\put(8.5,6.2){{\sc {\color{red} $\mathcal{E}_0$}}}
\put(10.3,2.7){{\sc {\color{red} $\mathcal{E}_1^*$}}}
\put(8.5,4.2){{\sc {\color{blue} $\mathcal{E}_2^*$}}}
\end{picture}
\begin{picture}(8,5)(0,0)
\put(-3.5,0){\rotatebox{0}{\includegraphics[width=8cm,height=9cm]{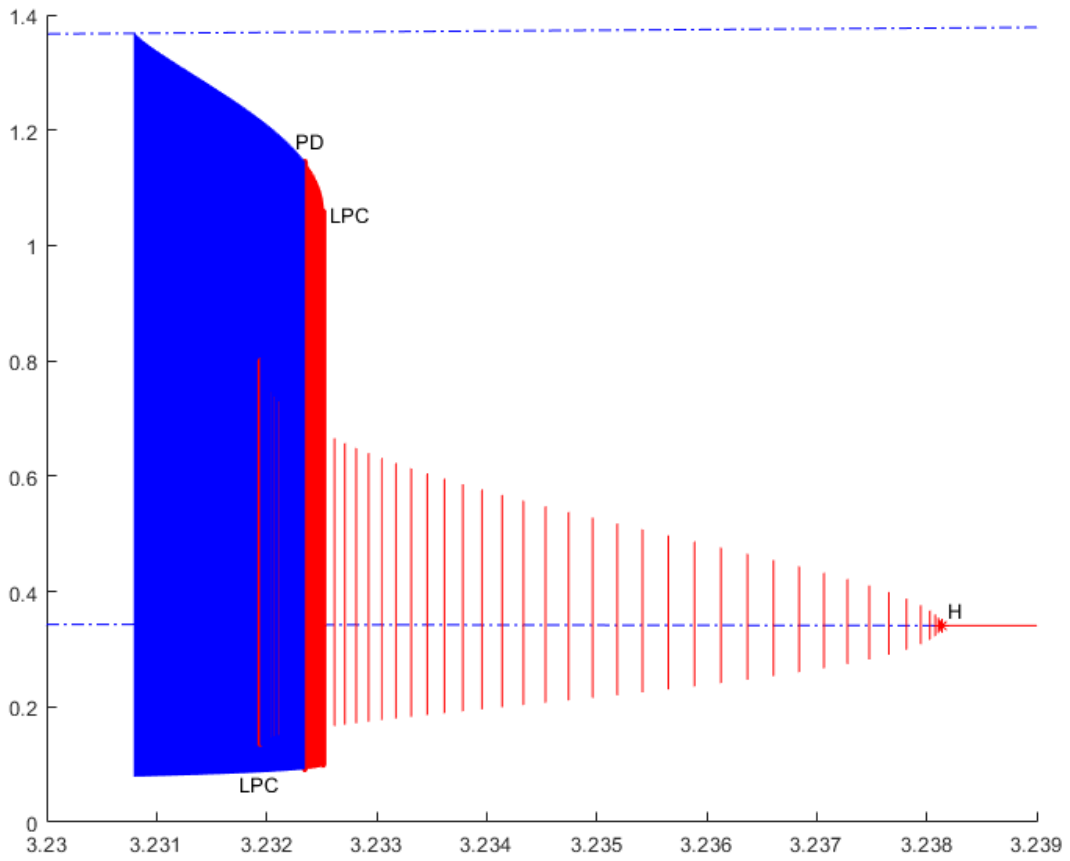}}}
\put(3.5,0){\rotatebox{0}{\includegraphics[width=8cm,height=9cm]{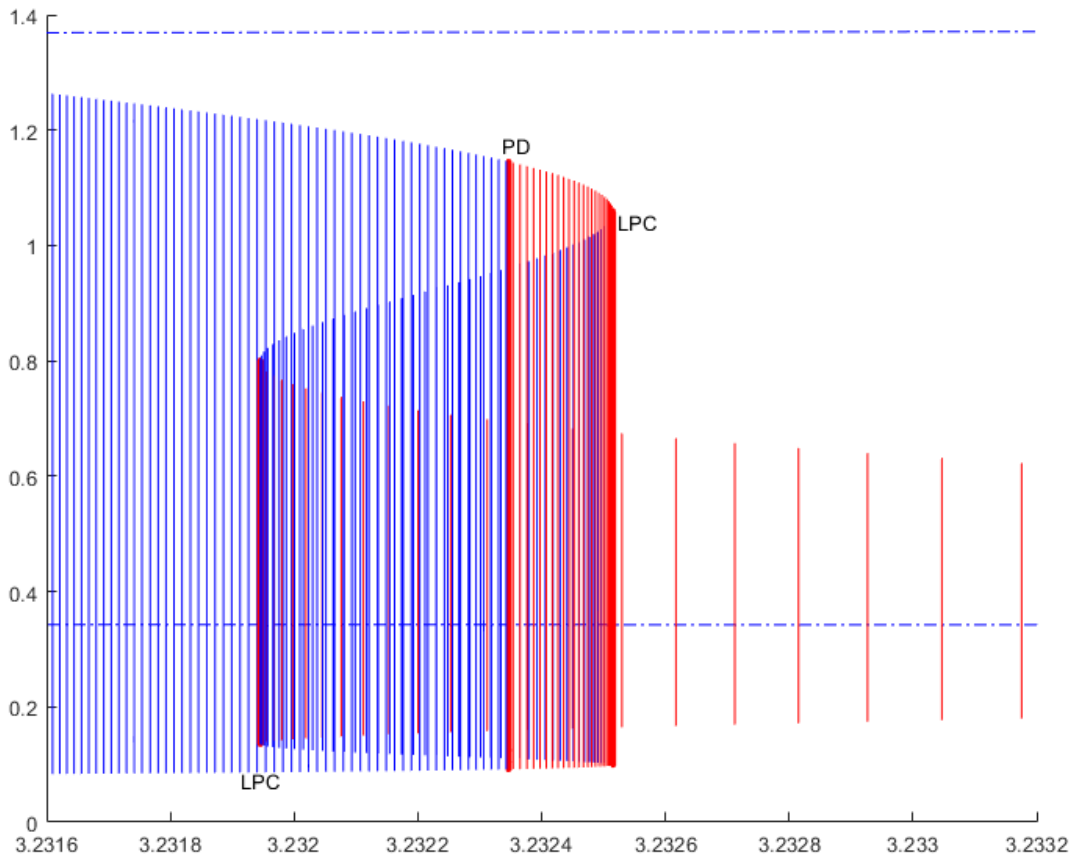}}}
\put(0.2,6.8){{\sc $(c)$}}
\put(-2.5,6.7){{\sc $S$}}
\put(3.55,2.5){{\sc  $S_{in}$}}
\put(-2.4,3.65){{\sc {\color{blue} $\mathcal{E}_1^*$}}}
\put(-2.1,6.65){{\tiny Hom}}
\put(-1.4,6.25){{\sc {\color{blue} $C_3$}}}
\put(-0.8,5.8){{\sc {\color{red} $C_3$}}}
\put(1.6,4.05){{\sc {\color{red} $C_1$}}}
\put(3.55,3.5){{\sc {\color{red} $\mathcal{E}_1^*$}}}
\put(1.5,6.3){{\sc {\color{blue} $\mathcal{E}_2^*$}}}
\put(7.2,6.8){{\sc $(d)$}}
\put(4.5,6.7){{\sc $S$}}
\put(10.9,2.5){{\sc $S_{in}$}}
\put(5.8,6.2){{\sc {\color{blue} $C_3$}}}
\put(7.7,5.9){{\sc {\color{red} $C_3$}}}
\put(9.5,4.5){{\sc {\color{red} $C_1$}}}
\put(6.6,5.35){{\sc {\color{blue} $C_2$}}}
\put(9.85,3.65){{\sc {\color{blue} $\mathcal{E}_1^*$}}}
\put(8.5,6.3){{\sc {\color{blue} $\mathcal{E}_2^*$}}}
\end{picture}
\end{center}
\vspace{-2.5cm}
\caption{\textsc{MatCont} results for $D=0.2$. (a) One-parameter bifurcation diagram of the substrate $S$ with respect to $S_{in}$, together with the associated branches of periodic solutions. (b)--(d) Enlarged views near the bifurcation points, highlighting the LP, LPC, PD, H, and Hom bifurcations. Stable solutions are shown in red and unstable ones in blue.}\label{Fig-DB-D02}
\end{figure}
\begin{proof}
For $D=0.2$, the horizontal line in the operating diagram intersects the bifurcation curves 
$\Gamma_{LP}$, $\Gamma_H$, and $\Gamma_{LPC}$ (twice) at the critical values $\sigma_i$ of $S_{in}$ listed in \cref{Tab-DB02}, which determine the sequence of bifurcations of equilibria and limit cycles.

More precisely, from \Cref{PropStabE0}, the washout equilibrium $\mathcal{E}_0$ is always LES. Starting from $S_{in}=0$, a limit point (LP) bifurcation occurs at $S_{in}=\sigma_1$, creating two positive equilibria, $\mathcal{E}_1^*$ and $\mathcal{E}_2^*$. 
While $\mathcal{E}_2^*$ remains unstable, $\mathcal{E}_1^*$ is LES for $S_{in}>\sigma_6$. 
At $S_{in}=\sigma_6$, $\mathcal{E}_1^*$ undergoes a supercritical Hopf bifurcation ($l_1<0$): when $S_{in}$ decreases through $\sigma_6$, it loses stability and a stable limit cycle $C_1$ emerges.
Following the branch of $C_1$ with decreasing $S_{in}$, it collides with an unstable limit cycle $C_2$ at $\sigma_3$ through a Limit Point of Cycles (LPC) bifurcation, causing both cycles to disappear.  
Continuing with increasing $S_{in}$ from $\sigma_3$, the unstable cycle $C_2$ persists and interacts with a newly created stable limit cycle $C_3$ at the second LPC bifurcation at $\sigma_5$, leading to the disappearance of both cycles.  
Decreasing $S_{in}$ from $\sigma_5$, $C_3$ undergoes a period-doubling (PD) bifurcation at $\sigma_4$, becoming unstable and producing a period-doubled orbit.  
Finally, as $S_{in}$ decreases further, the remaining unstable cycle grows in amplitude and vanishes in a homoclinic bifurcation (Hom) at $\sigma_2$, where the oscillation period diverges to $+\infty$ (see \cref{Fig-TnearHom}(b)).  
\end{proof}

The preceding analysis allows us to summarize the qualitative dynamics of the system as $S_{in}$ varies:

\begin{itemize}[leftmargin=*]
    \item For $S_{in} < \sigma_2$, the washout equilibrium $\mathcal{E}_0$ is the unique local attractor.
 
    \item For $S_{in} \in (\sigma_2, \sigma_3)$, only the unstable limit cycle $C_3$ exists, so $\mathcal{E}_0$ remains the sole attractor.

    \item For $S_{in} \in (\sigma_3, \sigma_4)$, one stable limit cycle ($C_1$) coexists with two unstable cycles ($C_2$ and $C_3$), leading to bistability between $\mathcal{E}_0$ and $C_1$.

    \item For $S_{in} \in (\sigma_4, \sigma_5)$, the system exhibits \emph{tri-stability}: trajectories may converge to $\mathcal{E}_0$, $C_1$, or $C_3$, depending on the initial condition.

    \item For $S_{in} \in (\sigma_5, \sigma_6)$, only the stable limit cycle $C_1$ persists among periodic solutions, resulting again in bistability between $\mathcal{E}_0$ and $C_1$.

    \item Finally, for $S_{in} > \sigma_6$, the positive equilibrium $\mathcal{E}_1^*$ becomes LES while $\mathcal{E}_0$ remains stable, yielding bistability between the two equilibria.
\end{itemize}
\subsection{Bifurcation scenarios for increasing dilution rate}
We now describe how the bifurcation structure evolves for larger values of the dilution rate $D$.  
In contrast to the case $D=0.2$, the configuration of Hopf, LPC, PD, and homoclinic bifurcations simplifies progressively as $D$ increases.  
We summarize below the qualitative scenarios corresponding to three representative values of $D$.
\begin{proposition}
Fix $D\in\{0.220,\,0.224,\,0.230\}$. For the specific growth functions $f_1$ and $f_2$ and the biological parameter values used in \cref{FigDO}, the one–parameter bifurcation diagrams of system \cref{ModelMutualism} with respect to $S_{in}$ are shown in \cref{Fig-DB3Cases}. These diagrams illustrate the equilibrium branches and the branches of periodic solutions together with the associated local bifurcations as $S_{in}$ varies. The critical values of $S_{in}$, as well as the existence and local stability of the corresponding steady states and periodic solutions for each value of $D$, are summarized in \cref{Tab-DB22-23}.
\end{proposition}
\begin{figure}[!ht]
\setlength{\unitlength}{1.0cm}
\begin{center}
\begin{picture}(7.9,8.7)(0,0)
\put(-5.2,-0.2){\rotatebox{0}{\includegraphics[width=7.5cm,height=13cm]{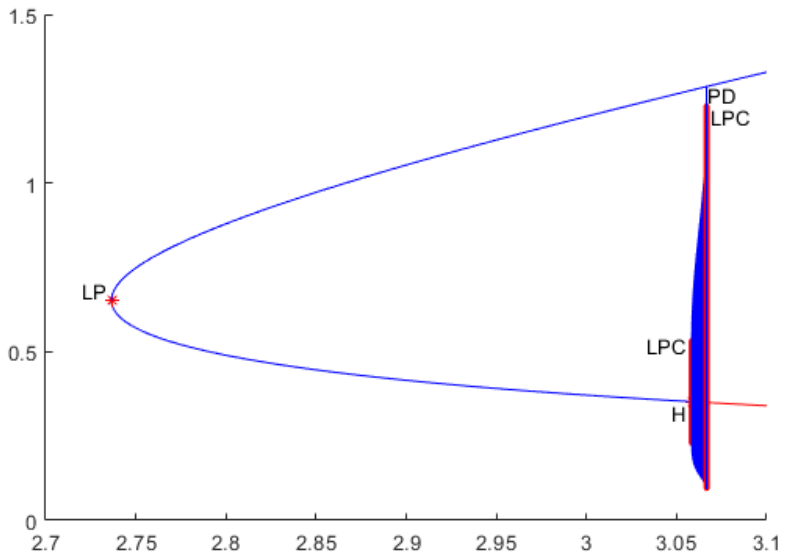}}}
\put(0,0){\rotatebox{0}{\includegraphics[width=7.5cm,height=13cm]{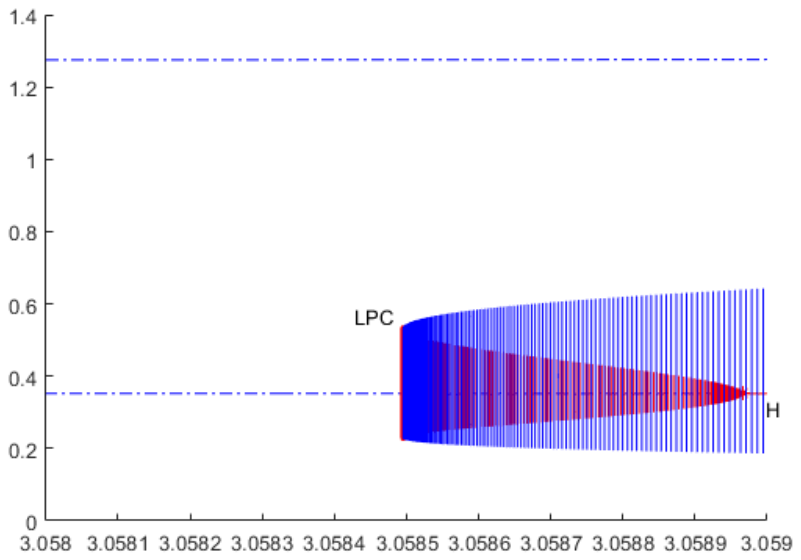}}}
\put(5.3,0){\rotatebox{0}{\includegraphics[width=7.5cm,height=13cm]{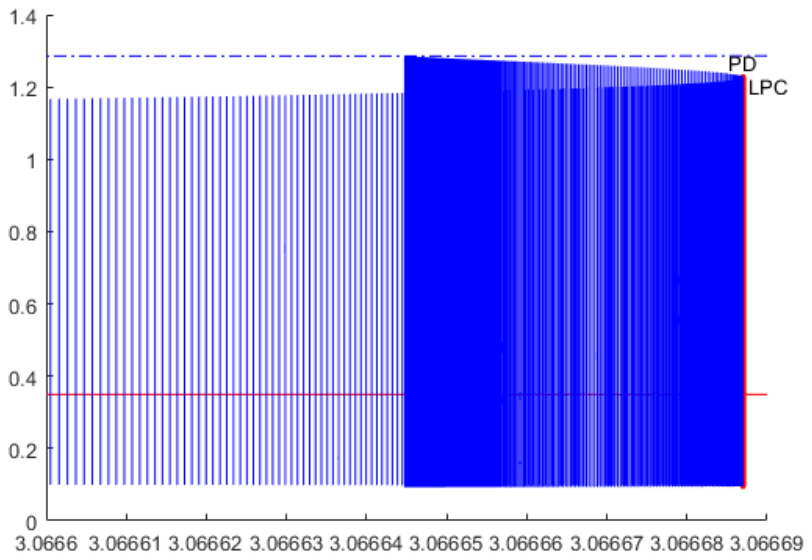}}}
\put(-1.7,8.3){{\sc $(a)$}}
\put(-3.4,8.3){{\sc $S$}}
\put(-1.2,7.6){{\sc {\color{blue} $\mathcal{E}_2^*$}}}
\put(-1.2,5.4){{\sc {\color{blue} $\mathcal{E}_1^*$}}}
\put(0.85,5.45){{\sc {\color{red} $\mathcal{E}_1^*$}}}
\put(0.85,4.65){{\sc  $S_{in}$}}
\put(3.7,8.45){{\sc $(b)$}}
\put(1.8,8.4){{\sc $S$}}
\put(2.5,7.85){{\sc {\color{blue} $\mathcal{E}_2^*$}}}
\put(2.5,5.4){{\sc {\color{blue} $\mathcal{E}_1^*$}}}
\put(4.7,6.45){{\sc {\color{blue} $C_2$}}}
\put(3.8,4.9){{\sc {\color{red} $C_1$}}}
\put(4,5.2){\sc {\color{red} $\nearrow$}}
\put(6.05,5.65){{\sc {\color{red} $\mathcal{E}_1^*$}}}
\put(6.05,4.7){{\sc  $S_{in}$}}
\put(9,8.4){{\sc $(c)$}}
\put(7.05,8.4){{\sc $S$}}
\put(7.6,8.2){{\sc {\color{blue} $\mathcal{E}_1^*$}}}
\put(8.3,5.8){{\sc {\color{red} $\mathcal{E}_1^*$}}}
\put(8,7.9){{\sc {\color{blue} $C_2$}}}
\put(10,8.15){{\sc {\color{blue} $C_3$}}}
\put(11.35,4.7){{\sc  $S_{in}$}}
\end{picture}
\\
\begin{picture}(7.9,4.4)(0,0)
\put(-5.2,0){\rotatebox{0}{\includegraphics[width=7.5cm,height=13cm]{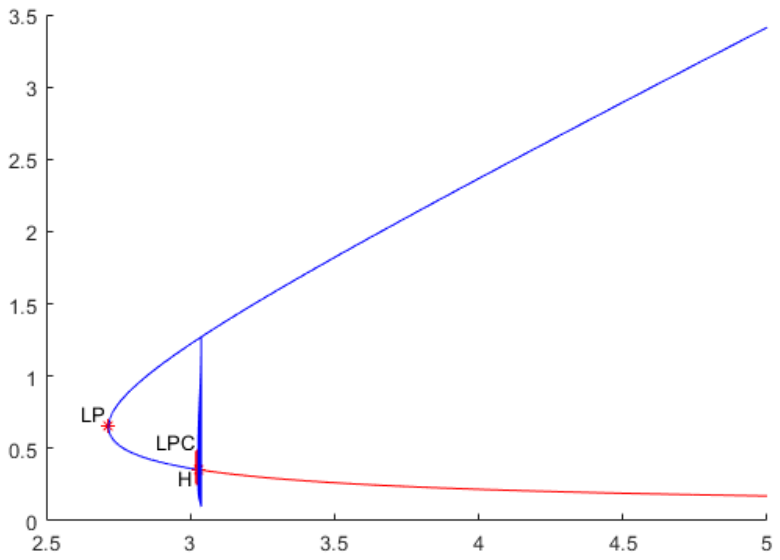}}}
\put(0,0){\rotatebox{0}{\includegraphics[width=7.5cm,height=13cm]{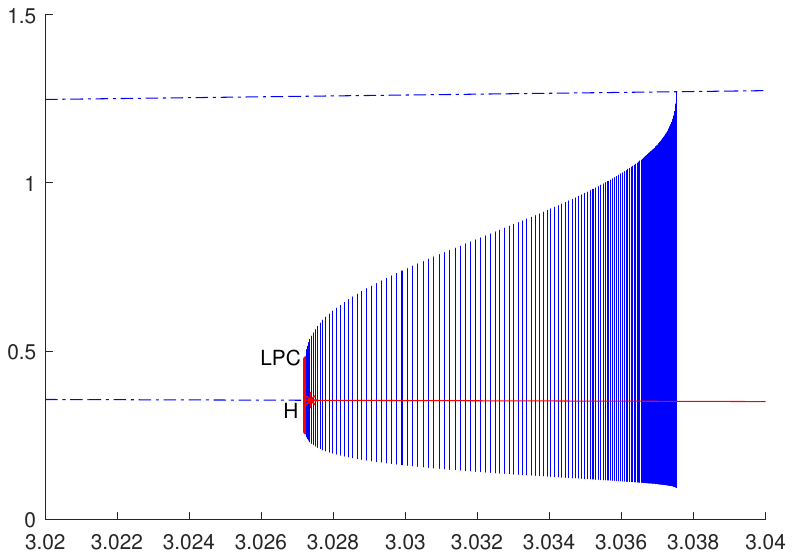}}}
\put(5.2,0){\rotatebox{0}{\includegraphics[width=7.5cm,height=13cm]{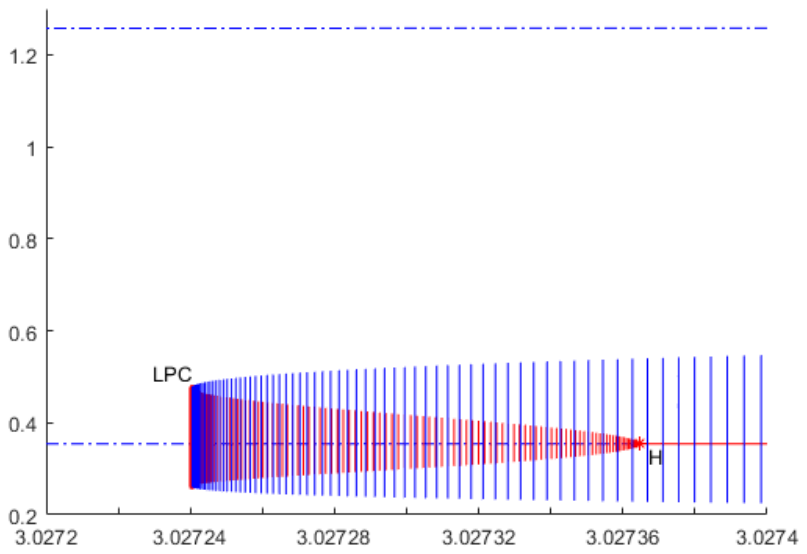}}}
\put(-1.8,8.3){{\sc $(d)$}}
\put(-3.4,8.3){{\sc $S$}}
\put(-1.2,7.35){{\sc {\color{blue} $\mathcal{E}_2^*$}}}
\put(-3.2,4.95){{\sc {\color{blue} $\mathcal{E}_1^*$}}}
\put(-1.2,5.1){{\sc {\color{red} $\mathcal{E}_1^*$}}}
\put(0.85,4.7){{\sc  $S_{in}$}}
\put(3.7,8.3){{\sc $(e)$}}
\put(1.75,8.3){{\sc $S$}}
\put(2.5,8){{\sc {\color{blue} $\mathcal{E}_2^*$}}}
\put(2.1,5.35){{\sc {\color{blue} $\mathcal{E}_1^*$}}}
\put(4.7,7.25){{\sc {\color{blue} $C_2$}}}
\put(6.05,5.5){{\sc {\color{red} $\mathcal{E}_1^*$}}}
\put(6.05,4.7){{\sc  $S_{in}$}}
\put(8.9,8.5){{\sc $(f)$}}
\put(6.95,8.4){{\sc $S$}}
\put(7.5,8.05){{\sc {\color{blue} $\mathcal{E}_2^*$}}}
\put(7,5.4){{\sc {\color{blue} $\mathcal{E}_1^*$}}}
\put(8.2,5.6){{\sc {\color{red} $C_1$}}}
\put(11.25,5.2){{\sc {\color{red} $\mathcal{E}_1^*$}}}
\put(9.5,5.9){{\sc {\color{blue} $C_2$}}}
\put(11.25,4.7){{\sc  $S_{in}$}}
\end{picture}
\\
\begin{picture}(7.9,4.5)(0,0)
\put(-3,0){\rotatebox{0}{\includegraphics[width=7.5cm,height=13cm]{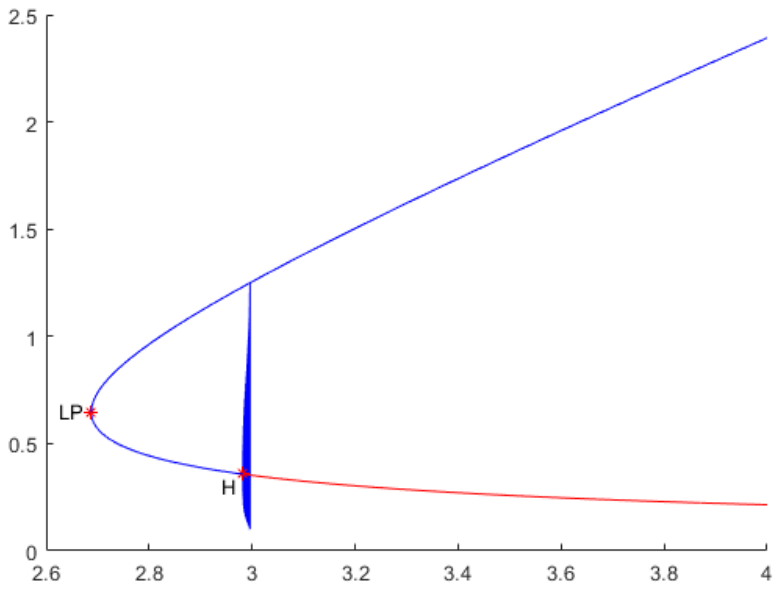}}}
\put(3,0){\rotatebox{0}{\includegraphics[width=7.5cm,height=13cm]{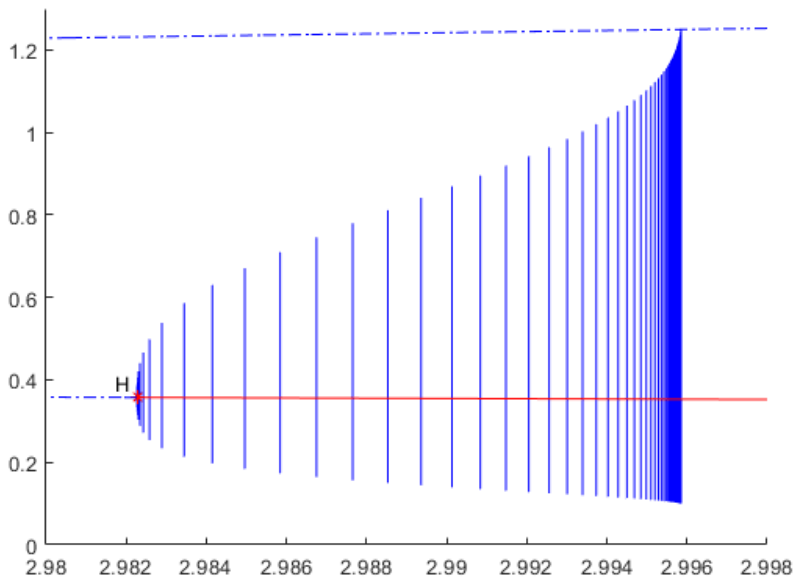}}}
\put(0.5,8.5){{\sc $(g)$}}
\put(-1.2,8.5){{\sc $S$}}
\put(1.2,7.65){{\sc {\color{blue} $\mathcal{E}_2^*$}}}
\put(-0.8,5){{\sc {\color{blue} $\mathcal{E}_1^*$}}}
\put(1.2,4.8){{\sc {\color{red} $\mathcal{E}_1^*$}}}
\put(3.1,4.6){{\sc  $S_{in}$}}
\put(6.5,8.6){{\sc $(h)$}}
\put(4.8,8.5){{\sc $S$}}
\put(5.5,8.1){{\sc {\color{blue} $\mathcal{E}_2^*$}}}
\put(7.2,7.5){{\sc {\color{blue} $C_1$}}}
\put(4.7,5.8){{\sc {\color{blue} $\mathcal{E}_1^*$}}}
\put(9.1,5.65){{\sc {\color{red} $\mathcal{E}_1^*$}}}
\put(9.1,4.6){{\sc  $S_{in}$}}
\end{picture}
\end{center}
\vspace{-4.5cm}
\caption{One-parameter bifurcation diagrams of system \cref{ModelMutualism} with respect to $S_{in}$. 
Panels (a--c) correspond to $D=0.220$, (d--f) to $D=0.224$, and (g--h) to $D=0.230$. 
The diagrams display the equilibrium branches of the substrate $S$ together with the associated branches of periodic solutions. 
Local bifurcations of equilibria and periodic orbits are indicated, including H, LPC, Hom, and PD bifurcations.}\label{Fig-DB3Cases}
\end{figure}
\begin{table}[!ht]
\caption{Critical values $\sigma_i$ of the input substrate concentration $S_{in}$ for different dilution rates $D$. For each case, the table reports the bifurcation type, the corresponding equilibrium coordinates $(S,x_1,x_2)$ or the period $T$ of periodic solutions, and the existence and stability of equilibria and limit cycles. 
The columns $\mathcal{E}_0$, $\mathcal{E}_1^*$, and $\mathcal{E}_2^*$ correspond to equilibria (S: stable, U: unstable), while the column $(C_1,C_2,C_3)$ indicates the periodic orbits present.}
\label{Tab-DB22-23}
\begin{center}
\begin{tabular}{ @{\hspace{1mm}}l@{\hspace{2mm}} @{\hspace{2mm}}l@{\hspace{2mm}}  
                 @{\hspace{2mm}}l@{\hspace{2mm}} @{\hspace{2mm}}l@{\hspace{2mm}}   
                 @{\hspace{2mm}}l@{\hspace{2mm}} @{\hspace{2mm}}l@{\hspace{2mm}} 
                @{\hspace{2mm}}l@{\hspace{2mm}} @{\hspace{2mm}}l@{\hspace{1mm}} }  
\hline
 $S_{in}$ & Bif. & $(S,x_1,x_2)$ / $T$ 
& $\mathcal{E}_0$ & $\mathcal{E}_1^*$ & $\mathcal{E}_2^*$ 
& {\sl$(C_1,C_2,C_3)$} & Notes \\
\hline

\multicolumn{8}{c}{\textbf{$D=0.22$}}\\
\hline

$(0,\sigma_1)$ & -- & -- & S & & & & {\sl$\mathcal{J}_0$} \\

$\sigma_1\approx2.736$ & LP & {\sc $(0.652,0.197,0.150)$} 
& & & & & {\sl$\mathcal{E}_1^*=\mathcal{E}_2^*$} \\

$(\sigma_1,\sigma_2)$ & -- & -- & S & U & U & & {\sl$\mathcal{J}_1^0$} \\

 $\sigma_2\approx3.0585$ & LPC & $T=9.60$ 
& & & & & {\sl Collision of $C_1$ and $C_2$} \\

 $(\sigma_2,\sigma_3)$ & -- & -- & S & U & U & $(S,U,\cdot)$ & {\sl$\mathcal{J}_1^{C_{12}^{su}}$} \\

$\sigma_3\approx3.0590$ & H & {\sc $(0.352,0.247,0.200)$} 
& & & & & {\sl Hopf ($l_1\approx-0.085$)} \\

 $(\sigma_3,\sigma_4)$ & -- & -- & S & S & U & $(\cdot,U,\cdot)$ & {\sl$\mathcal{J}_2^{C_2^u}$} \\

 $\sigma_4\approx3.066645$ & Hom & $T\to+\infty$ 
& & & & & {\sl Homoclinic} \\

 $(\sigma_4,\sigma_5)$ & -- & -- & S & S & U & $(\cdot,U,U)$ & {\sl $\mathcal{J}_2^{C_{23}^{uu}}$} \\

 $\sigma_5\approx3.066687$ & PD & $T=26.26$ 
& & & & &  {\sl Period-Doubling} \\

 $(\sigma_5,\sigma_6)$ & -- & -- & S & S & U & $(\cdot,U,S)$ & {\sl $\mathcal{J}_2^{C_{23}^{us}}$} \\

 $\sigma_6\approx3.066688$ & LPC & $T=26.13$ 
& & & & & {\sl Collision of $C_2$ and $C_3$} \\

 $(\sigma_6,+\infty)$ & -- & -- & S & S & U & & {\sl $\mathcal{J}_2^0$} \\

\hline
\multicolumn{8}{c}{\textbf{$D=0.224$}}\\
\hline

 $(0,\sigma_1)$ & -- & -- & S & & & & {\sl $\mathcal{J}_0$} \\

 $\sigma_1\approx2.716$ & LP & {\sc $(0.650,0.198,0.151)$} 
& & & & & {\sl $\mathcal{E}_1^*=\mathcal{E}_2^*$} \\

 $(\sigma_1,\sigma_2)$ & -- & -- & S & U & U & & {\sl $\mathcal{J}_1^0$} \\

 $\sigma_2\approx3.02724$ & LPC & $T=9.3$ 
& & & & & {\sl Collision of $C_1$ and $C_2$} \\

 $(\sigma_2,\sigma_3)$ & -- & -- & S & U & U & $(S,U,\cdot)$ & {\sl $\mathcal{J}_1^{C_{12}^{su}}$} \\

 $\sigma_3\approx3.02736$ & H & {\sc $(0.354,0.247,0.201)$} 
& & & & & {\sl Hopf ($l_1\approx-0.044$)} \\

 $(\sigma_3,\sigma_4)$ & -- & -- & S & S & U & $(\cdot,U,\cdot)$ & {\sl $\mathcal{J}_2^{C_2^u}$} \\

 $\sigma_4\approx3.0376$ & Hom & $T\to+\infty$ 
& & & & & {\sl Homoclinic} \\

 $(\sigma_4,++\infty)$ & -- & -- & S & S & U & & {\sl $\mathcal{J}_2^0$} \\

\hline
\multicolumn{8}{c}{\textbf{$D=0.23$}}\\
\hline

 $(0,\sigma_1)$ & -- & -- & S & & & & {\sl $\mathcal{J}_0$} \\

$\sigma_1\approx2.736$ & LP & {\sc $(0.647,0.199,0.152)$} 
& & & & & {\sl $\mathcal{E}_1^*=\mathcal{E}_2^*$} \\

 $(\sigma_1,\sigma_2)$ & -- & -- & S & U & U & & {\sl $\mathcal{J}_1^0$} \\

 $\sigma_2\approx2.982$ & H & {\sc $(0.357,0.248,0.201)$} 
& & & & & {\sl  Hopf ($l_1\approx0.017$)} \\

 $(\sigma_2,\sigma_3)$ & -- & -- & S & S & U & $(U,\cdot,\cdot)$ & {\sl $\mathcal{J}_2^{C_1^u}$} \\

 $\sigma_3\approx2.996$ & Hom & $T\to+\infty$ 
& & & & & {\sl Homoclinic} \\

 $(\sigma_3,+\infty)$ & -- & -- & S & S & U & & {\sl $\mathcal{J}_2^0$} \\
\hline
\end{tabular}
\end{center}
\end{table}
\begin{proof}
For each $D \in \{0.220, 0.224, 0.230\}$, the bifurcation diagrams of system \cref{ModelMutualism} with respect to $S_{in}$ are obtained by continuation of equilibria and periodic solutions. 
Local bifurcations (LP, H, LPC, Hom, PD) and the associated stability are identified along the branches. 
As $D$ increases, secondary bifurcations progressively disappear, leading to simpler periodic structures. 
All critical values of $S_{in}$ and stability properties are reported in \cref{Tab-DB22-23}, and the diagrams are shown in \cref{Fig-DB3Cases}.
\end{proof}

The qualitative dynamics of the system with respect to $S_{in}$ can be summarized as follows for the three values of $D$ (see \cref{Fig-DB3Cases} and \cref{Tab-DB22-23}):

\begin{itemize}[leftmargin=*]
\item For $D=0.220$, the system exhibits rich multistability. As $S_{in}$ increases, two stable limit cycles, $C_1$ and $C_3$, emerge through Hopf and LPC bifurcations, while $\mathcal{E}_0$ remains LES. Secondary bifurcations such as PD and LPC mediate transitions between coexisting periodic orbits, resulting in bistability or tristability depending on $S_{in}$.

\item For $D=0.224$, the dynamics simplifies. A Hopf bifurcation generates a single stable limit cycle $C_1$, which interacts with an unstable cycle $C_2$ via an LPC bifurcation. As $S_{in}$ increases further, the limit cycles disappear through a homoclinic bifurcation. Over the intermediate range, the system may exhibit bistability either between $\mathcal{E}_0$ and $C_1$ or between $\mathcal{E}_0$ and the positive equilibrium $\mathcal{E}_1^*$, with no secondary bifurcations occurring.

\item For $D=0.230$, the dynamics is even simpler. A subcritical Hopf bifurcation gives rise to an unstable periodic orbit, while for larger $S_{in}$, the positive equilibrium $\mathcal{E}_1^*$ becomes LES. This results in bistability between $\mathcal{E}_0$ and $\mathcal{E}_1^*$. No secondary bifurcations are present in this case.
\end{itemize}

Overall, increasing $D$ reduces multistability and coexisting periodic solutions, simplifying global dynamics. Bistability regions shrink, and the system transitions from complex behavior with multiple limit cycles to simpler dynamics dominated by a single stable equilibrium or limit cycle, highlighting the key role of dilution rate in shaping coexistence patterns.
\section{Conclusions}                                 \label{Sec-Conc}
In this work, we analyzed a chemostat model describing the competition of two species for a single nutrient under an obligate mutualistic interaction, where each species promotes and depends on the growth of the other. Mortality was incorporated through distinct removal rates for each species.  

We first carried out an analytical study of the system, determining the existence of equilibria and establishing local stability conditions for the different equilibria. We then proved the global asymptotic stability of the washout equilibrium $\mathcal{E}_0$.
The model admits two classes of equilibria: the washout state $\mathcal{E}_0$, which always exists, and one or several coexistence equilibria $\mathcal{E}^*$. A coexistence equilibrium exists if and only if the curves $\gamma_1$ and $\gamma_2$, associated with the functions $F_1$ and $F_2$, intersect at positive values. Its local stability is determined using the Routh–Hurwitz criterion \cref{Expr-CRH}, since the eigenvalues of the Jacobian cannot be obtained explicitly. We show that $\mathcal{E}^*$ may lose stability through a supercritical Hopf bifurcation when condition $c_4$ \cref{Expc4} is not always satisfied.

Using numerical continuation with \textsc{MatCont} \cite{MATCONT2023}, we explored the operating diagram of system \cref{ModelMutualism} in the $(S_{in},D)$ parameter plane. In the case without mortality ($D_1=D_2=D$), positive equilibria appear or disappear only through Limit Point (saddle-node) bifurcations. By contrast, introducing distinct removal rates ($D_i=\alpha_i D+a_i$, with $\alpha_i>0$ and $a_i\neq 0$) enriches the system dynamics significantly: additional bifurcations arise, including Hopf bifurcations and Limit Points of Cycles (LPC). 
We also identified several codimension-two bifurcations, including Bogdanov--Takens (BT), cusp bifurcations of cycles (CPC), 1:1 (R1) and 1:2 (R2) resonances, and generalized Hopf (GH, Bautin) bifurcations. These bifurcations are difficult to obtain analytically but were detected numerically. They organize the branches of periodic solutions and induce changes in their stability, enriching the range of possible dynamical behaviors.

The one-parameter bifurcation diagrams in $S_{in}$ (for fixed $D$) further illustrate the qualitative transitions in equilibrium and periodic solutions as the system moves across different regions of the operating diagram (\cref{FigDO}), revealing the emergence of stable and unstable limit cycles and the possible coexistence of multiple attractors.
In particular, we observed tri-stability, where trajectories may converge either to the washout equilibrium or to one of two stable limit cycles (one emerging from a Hopf bifurcation and the other from an LPC bifurcation), depending on initial conditions. Consequently, coexistence may occur near a positive equilibrium or along a stable periodic orbit.  

Overall, our study demonstrates that incorporating distinct mortality rates profoundly influences the dynamics of obligate mutualism in chemostats. It generates a wide spectrum of bifurcation phenomena and multistability scenarios, highlighting the complex interplay between species interactions and mortality in microbial ecosystems.

An important message of this work is that mortality cannot be neglected: it significantly enriches the system dynamics, leading to stable periodic solutions, tri-stability, limit point of cycles (LPC) and period-doubling (PD) bifurcations, as well as generalized Hopf (GH) bifurcations and homoclinic bifurcations.
These behaviors provide a more coherent picture consistent with natural observations, where multiple species may coexist around stable limit cycles. In contrast, models neglecting mortality, such as those in El Hajji \cite{ElhajjiIJB2018,ElhajjiJBD2009}, predict coexistence only near positive equilibrium points and fail to capture this richness of dynamics.
\appendix
\section{The case of the growth functions \cref{SpeciFunc}} \label{AppendixA}
In this section, we analyze the particular case where the growth functions are given by \cref{SpeciFunc}. Our goal is to characterize the maximal number of solutions of equation \cref{EquExis_xji} on the biologically relevant interval.
The following result shows that, under condition \cref{CondExis_xji}, this number is exactly two.
\begin{proposition}
Let the growth functions be given by \cref{SpeciFunc} and assume that condition \cref{CondExis_xji} holds. 
Then, for each $i,j \in \{1,2\}$ with $i\neq j$, equation \cref{EquExis_xji} has exactly two solutions $x_j^1$ and $x_j^2$ in the interval $]0,DS_{in}/D_j[$.
\end{proposition}
\begin{proof}
Let the growth functions be given by \cref{SpeciFunc} and recall the definition 
of $\phi_i$ in \cref{FunPhi_i}, for $i,j \in \{1,2\}$ with $i\neq j$. 
Introducing the notation $\theta_j :=D_j/D$, a straightforward computation yields
$$
\phi_i'(x_j) 
= m_i \,\frac{N(x_j)}{(K_i+S_{in}-\theta_j x_j)^2(L_i+x_j)^2},
$$
where the numerator $N(x_j)$ is a quadratic polynomial :
$$
N(x_j) = \big(\theta_j^2 L_i - \theta_j K_i\big)x_j^2 - 2\theta_j L_i(K_i+S_{in})\,x_j + S_{in}L_i(K_i+S_{in}).
$$
Its reduced discriminant is given by
$$
\Delta'=\theta_j K_i L_i (K_i+S_{in})(\theta_j L_i+S_{in}),
$$
which is positive since all parameters are positive. 
Hence $N(x_j)$ has two distinct real roots and therefore the equation $\phi_i'(x_j)=0$ has at most two solutions in the interval $]0,DS_{in}/D_j[$.
According to \cref{hyp1,hyp2,hyp3}, the function $\phi_i(x_j)$ is positive, continuous, non-monotonic on $]0, DS_{in}/D_j[$ and satisfies $\phi_i(0) = \phi_i(DS_{in}/D_j) = 0$.
Therefore, the function $\phi_i(x_j)$ necessarily has a unique positive maximum in $]0, DS_{in}/D_j[$.
Under condition \cref{CondExis_xji}, the horizontal line $y=D_i$ intersects the graph of $\phi_i$ at exactly two points $x_j^1$ and $x_j^2$ in $]0,DS_{in}/D_j[$ (see \cref{Fig-NmbExSol_xji}).
\end{proof}

The following result characterizes the extrema of the function $F_i(x_j)$ when the growth functions are specified by \cref{SpeciFunc}.
\begin{proposition}
Let the growth functions be given by \cref{SpeciFunc}. 
Then, for each $i,j \in \{1,2\}$ with $i\neq j$, the function $x_j \mapsto F_i(x_j)$ 
admits a unique extremum at some point $x_j^* \in ]x_j^1,x_j^2[$,
where $x_j^1$ and $x_j^2$ denote the two positive solutions of 
$f_i(S_{in}-D_jx_j/D,x_j)=D_i$.
\end{proposition}
\begin{proof}
Consider the specific growth functions \cref{SpeciFunc}. 
The function $F_i$ defined in \cref{LemFi} can be written as
$$
F_i(x_j) \;=\; 
\frac{D_i\,(K_i + S_{in} - \theta_j x_j)(L_i + x_j) - m_i x_j \,(S_{in} - \theta_j x_j)}
{\theta_i \big((D_i-m_i)x_j + D_i L_i\big)},
$$
where $\theta_j := D_j / D$. A straightforward calculation of the derivative yields
$$
F'_i(x_j) = \frac{- \theta_j (D_i - m_i)^2 x_j^2 - 2 \, \theta_j D_i L_i (D_i - m_i) x_j + D_i L_i (m_i K_i - D_i \theta_j L_i)}{D_i \big((D_i-m_i)x_j + D_i L_i\big)^2}.
$$
The numerator is a quadratic polynomial in $x_j$ with reduced discriminant
$$
\Delta' = \theta_j\, (D_i-m_i)^2 D_i\, L_i\, K_i \, m_i  > 0,
$$
so it has two distinct real roots. Hence, $F'_i(x_j)=0$ has at most two solutions in the interval $]x_j^1, x_j^2[$.
By \cref{LemFi}, the function $F_i(x_j)$ is positive and continuous on $]x_j^1, x_j^2[$, and vanishes at the endpoints $F_i(x_j^1)=F_i(x_j^2)=0$. Therefore, it attains a unique positive maximum in this interval.  
\end{proof}
\section{Parameter values used in this paper}        \label{AppendixC}
All parameter values used in the numerical simulations are listed in \cref{Tab-Allpar}. 
The software \textsc{Maple} was used to generate \cref{Fig-NmbExSol_xji,FigVP}, 
while \textsc{MatCont} was used to produce 
\cref{FigDOSSMort,FigDO,Fig-DBD0195,Fig-TnearHom,Fig-DB-D02,Fig-DB3Cases}.
\begin{table}[ht]
\caption{Parameter values used for model \cref{ModelMutualism} when the growth rates $f_1$ and $f_2$ are given by \cref{SpeciFunc}.}   \label{Tab-Allpar}
\vspace{-0.2cm}
\centering{
\begin{tabular}{ @{\hspace{1mm}}l@{\hspace{2mm}} @{\hspace{2mm}}l@{\hspace{2mm}} @{\hspace{2mm}}l@{\hspace{2mm}} @{\hspace{2mm}}l@{\hspace{2mm}}
                 @{\hspace{2mm}}l@{\hspace{2mm}} @{\hspace{2mm}}l@{\hspace{2mm}} @{\hspace{2mm}}l@{\hspace{2mm}} @{\hspace{2mm}}l@{\hspace{2mm}}
                 @{\hspace{2mm}}l@{\hspace{2mm}} @{\hspace{2mm}}l@{\hspace{2mm}} @{\hspace{2mm}}l@{\hspace{2mm}}}
Parameter                     & $m_1$ & $K_1$ & $L_1$ & $m_2$  & $K_2$  &  $L_2$  &  $\alpha_1$  &  $\alpha_2$  &  $a_1$   &  $a_2$ \\ \hline
\begin{tabular}{l}
\cref{FigDO,Fig-DBD0195,Fig-TnearHom,Fig-DB-D02,FigVP,Fig-DB3Cases} \\
\cref{FigDOSSMort}
\end{tabular}
                    & 4 & 0.2 & 0.3 & 4 & 0.1 & 0.2 & 1 & 1 &
\begin{tabular}{l}
0.8   \\
0
\end{tabular}
                     &
\begin{tabular}{l}
1.5   \\
0
\end{tabular}
\end{tabular}}
\end{table}
\section*{Acknowledgments} 
The authors are grateful to Tewfik Sari for his valuable comments and suggestions, which greatly improved this work.
They also acknowledge the support of the Euro-Mediterranean research network \href{https://treasure.hub.inrae.fr/}{Treasure} and the Tunisian Ministry of Higher Education and Scientific Research (Young Researchers’ Encouragement Program: 06P1D2024-PEJC).
 
\end{document}